\documentclass[twocolumn]{article}

\PassOptionsToPackage{numbers, compress}{natbib}

\usepackage[utf8]{inputenc} \usepackage[T1]{fontenc}    \usepackage{hyperref}       \usepackage{url}            \usepackage{booktabs}       \usepackage{amsfonts}       \usepackage{nicefrac}       \usepackage{microtype}      \usepackage{amsthm,amssymb,amsmath}

\usepackage{graphicx} 
\usepackage{subfigure} 
\usepackage{natbib}
\usepackage{algorithm}
\usepackage{algpseudocode}
\usepackage{comment}
\usepackage{hyperref}

\usepackage[T1]{fontenc}
\usepackage{lmodern}
\usepackage{url}
\usepackage{verbatim}
\usepackage{amsmath}
\usepackage{xspace}
\usepackage{amssymb}
\usepackage{color}
\usepackage{graphicx}
\usepackage{stmaryrd}
\usepackage{bbm}
\usepackage{listings}  
\usepackage[bbgreekl]{mathbbol}

\newcommand*{\qedh}{\hfill\ensuremath{\square}}%

\DeclareSymbolFontAlphabet{\amsmathbb}{AMSb}

\newcommand{\reals}{\ensuremath{\amsmathbb{R}}}
\newcommand{\naturals}{\ensuremath{\amsmathbb{N}}}

\newcommand{\onevec}{\ensuremath{\mathbf{1}}}

\newcommand{\argmin}{\mathop{\arg\,\min}}
\newcommand{\trace}{\mathop{\mathsf{tr}}}

\newcommand{\Symmetric}{\amsmathbb{S}}
\newcommand{\Permutations}{\mathbb{P}}
\newcommand{\Ortho}{\mathbb{O}}
\newcommand{\DStoch}{\amsmathbb{W}}
\newcommand{\conv}{\mathtt{conv}}
\newcommand{\embed}{\psi}

\newcommand{\InnerDSLtwo}{\textbf{InnerDSL2}\xspace}
\newcommand{\NetAlignBP}{\textbf{NetAlignBP}\xspace}
\newcommand{\SparseIsoRank}{\textbf{SparseIsoRank}\xspace}
\newcommand{\IsoRank}{\textbf{IsoRank}\xspace}
\newcommand{\NetAlignMR}{\textbf{NetAlignMR}\xspace}
\newcommand{\Natalie}{\textbf{Natalie}\xspace}
\newcommand{\DSLone}{\textbf{DSL1}\xspace}
\newcommand{\DSLtwo}{\textbf{DSL2}\xspace}
\newcommand{\InnerPerm}{\textbf{InnerPerm}\xspace}
\newcommand{\InnerDSLone}{\textbf{InnerDSL1}\xspace}
\newcommand{\EXACT}{\textbf{EXACT}\xspace}

\newtheorem{lemma}{Lemma}

\newtheorem{theorem}{Theorem}
\newtheorem{rmk}{Remark}

\includecomment{supplement} 

\newcommand{\alt}[2]{#1}
\begin{supplement}
\renewcommand{\alt}[2]{#2}
\end{supplement}

\usepackage{array}

\newcounter{packednmbr}

\newenvironment{packeditemize}{\begin{list}{$\bullet$}{\setlength{\itemsep}{0pt}\addtolength{\labelwidth}{-5pt}\setlength{\leftmargin}{\labelwidth}\setlength{\listparindent}{\parindent}\setlength{\parsep}{0pt}\setlength{\topsep}{3pt}}}{\end{list}}

\usepackage[font={footnotesize}]{caption}

\begin{document}

\title{\Large A Family of Tractable Graph Distances}

\author{
  Jose Bento\thanks{Boston College, \texttt{jose.bento@bc.edu}}     \and
  Stratis Ioannidis\thanks{Northeastern University, \texttt{ioannidis@ece.neu.edu} } 
     }

\date{}

\maketitle

\begin{abstract}
Important data mining problems such as nearest-neighbor search and clustering admit theoretical guarantees when restricted to  objects  embedded in a metric space. Graphs are ubiquitous, and clustering and classification over graphs arise in diverse areas, including, e.g., image processing and social networks. Unfortunately, popular distance scores used in these applications, that scale over large graphs, are not metrics and thus come with no guarantees.   Classic graph distances such as, e.g., the chemical  and the CKS distance are arguably natural and intuitive, and are indeed also metrics, but they are  intractable: as such, their computation does not scale to large graphs. 
We define a broad family of graph distances, that includes both the chemical and the CKS distance, and prove that these are all metrics. Crucially, we show that our family  includes metrics that are tractable. Moreover, we extend these distances by incorporating auxiliary node attributes, which is important in practice, while maintaining both the metric property and  tractability.
\end{abstract}

\section{Introduction}
\label{sec:intro}
Graph similarity and the related problem of  graph isomorphism have a long history in data mining, machine learning, and pattern recognition \citep{conte2004thirty,macindoe2010graph,koutra2013deltacon}.  \emph{Graph distances} naturally arise in this literature: intuitively, given two (unlabeled) graphs, their distance is a score quanitifying their structural differences.
A highly desirable property for such a score is that it is a \emph{metric}, i.e., it is non-negative, symmetric, positive-definite, and, crucially, satisfies the triangle inequality. Metrics exhibit significant computational advantages over non-metrics. For example, operations such as nearest-neighbor search \citep{clarkson2006nearest,clarkson1999nearest,beygelzimer2006cover},  clustering \citep{ackermann2010clustering}, outlier detection \citep{angiulli2002fast}, and diameter computation \citep{indyk1999sublinear} admit fast algorithms precisely when performed over objects embedded in a metric space.  To this end, proposing \emph{tractable}  graph metrics is of paramount importance in applying such algorithms to graphs.

Unfortunately, graph metrics of interest are often computationally expensive.
A well-known example is the \emph{chemical distance} \citep{kvasnivcka1991reaction}. Formally, given graphs $G_A$ and $G_B$, represented by their adjacency matrices $A,B\in\{0,1\}^{n\times n}$, the  chemical distance is $d_{\Permutations^n} (A,B)$ is defined in terms of  a mapping between the two graphs that minimizes their edge discrepancies, i.e.:\begin{align}d_{\Permutations^n} (A,B) =\textstyle \min_{P\in\Permutations^n} \|AP-PB\|_F,\label{chemical}\end{align}
where $\Permutations^n$ is the set of permutation matrices of size $n$ and $\|\cdot\|_F,$ is the Frobenius norm (see Sec.~\ref{sec:technical} for definitions). The \emph{Chartrand-Kubiki-Shultz (CKS)} \citep{chartrand1998graph} distance is an alternative: CKS is again given by \eqref{chemical} but, instead of edges, matrices $A$ and $B$ contain the pairwise shortest path distances between any two nodes. 
The chemical and CKS distances have important properties. First, they are zero if and only if the graphs are isomorphic, which  appeals to both intuition and practice; second, as desired, they are metrics; third, they have a natural interpretation,  capturing global 
  structural similarities between graphs. However, finding an optimal permutation $P$ is notoriously hard; graph isomorphism, which is equivalent to deciding if there exists a permutation $P$ s.t.~$AP=PB$ (for both adjacency and path matrices), is famously a problem that is neither known to be in P nor shown to be NP-hard \citep{babai2016graph}.  
There is a large and expanding literature on scalable heuristics to estimate the optimal permutation $P$ \citep{klau2009new,bayati2009algorithms,lyzinski2016graph,el2015natalie}. Despite their computational advantages, unfortunately, using them to approximate $d_{\Permutations^n} (A,B)$ breaks the metric property.

This significantly degrades the performance of many important tasks that rely on computing distances
between graphs. For example, there is a clear separation on the approximability of clustering over metric and non-metric spaces \cite{ackermann2010clustering}. We also demonstrate this empirically in Section~\ref{sec:experiments} (c.f.~Fig.~\ref{fig:clustering_exp}): attempting to cluster graphs sampled from well-known families based on non-metric distances significantly increases the misclassification rate, compared to clustering using metrics. 

An additonal issue that arises in practice is that nodes often have attributes not associated with adjacency.  For example, in social networks, nodes may contain profiles with a user's age or gender; similarly, nodes in molecules may be labeled by atomic numbers. Such attributes are not captured by the chemical or CKS distances. However, in such cases, 
only \emph{label-preserving} permutations $P$ may make sense (e.g., mapping females to females,  oxygens to oxygens, etc.). Incorporating attributes  while preserving the metric property is thus important from a practical perspective.

\noindent\textbf{Contributions.} We seek  generalization of the chemical and CKS distances that (a) \emph{satisfy the metric property} and (b) are \emph{tractable}: by this, we mean that they can be computed either by solving a convex optimization problem, or by a polynomial time algorithm. Specifically, we study generalizations of \eqref{chemical} of the form:
\begin{align}d_S(A,B)=\textstyle \min_{P\in S}\|AP-PB\|\label{general} \end{align}
where $S \subset \reals^{n\times n}$ is closed and bounded, $\|\cdot\|$ is a matrix norm, and $A,B\in \reals^{n\times n}$ are arbitrary real matrices (representing adjacency, path distances, weights, etc.). We make the following contributions:
\begin{packeditemize}
\item We prove sufficient conditions on $S$ and norm $\|\cdot\|$ for which \eqref{general} is a metric.
In particular, we show that $d_S$ is a so-called \emph{pseudo-metric} (see Sec.~\ref{sec:technical}) when:
\begin{packeditemize}
 \item[{\em (i)}] $S=\Permutations^n$ and $\|\cdot\|$ is any entry-wise or operator norm;
\item[{\em (ii)}] $S=\DStoch^n$, the set of \emph{doubly stochastic} matrices, $\|\cdot\|$ is an arbitrary entry-wise norm, and $A,B$ are symmetric; a modification on $d_S$  extends this result to both operator norms as well as arbitrary matrices (capturing, e.g., directed graphs); and
\item[{\em (iii)}] $S=\Ortho^n$, the set of orthogonal matrices, and $\|\cdot\|$ is the operator or entry-wise 2-norm.
\end{packeditemize} Relaxations \emph{(ii)} and \emph{(iii)} are very important from a practical standpoint. For all matrix norms, computing \eqref{general} with $S=\DStoch^n$ is tractable, as it is a convex optimization. For $S=\Ortho^n$, \eqref{general} is non-convex but is still tractable, as it reduces to a spectral decomposition. This was known for the Frobenius norm \cite{umeyama1988eigendecomposition};  we prove this is the case for the operator 2-norm also.
\item We include node attributes in a natural way in the definition of $d_S$ as both \emph{soft} (i.e., penalties in the objective) or \emph{hard} constraints in Eq.~\eqref{general}. Crucially, we do this \emph{without affecting the metric property and tractability}. This allows us to explore label or feature preserving permutations, that incorporate both (a) exogenous node attributes, such as, e.g., user age or gender in a social network,  as well as (b)  endogenous, structural features of each node, such as its degree or the number of triangles that pass through it. We numerically show that adding these constraints can speed up the computation of $d_S$.
\end{packeditemize} 

From an experimental standpoint, we extensively compare our tractable metrics to several existing heuristic approximations. We also demonstrate the tractability of our metrics
by parallelizing their execution using the alternating method of multipliers \citep{boyd2011distributed}, which we implement over a compute cluster using Apache Spark \citep{zaharia2010spark}.

\noindent\textbf{Related Work.} Graph distance (or similarity) scores find applications in varied fields such as in image processing \citep{conte2004thirty}, chemistry~\citep{allen2002cambridge,kvasnivcka1991reaction}, and social network analysis \citep{macindoe2010graph,koutra2013deltacon}.
 Graph distances are easy to define when, contrary to our setting, the correspondence between graph nodes is known, i.e., graphs are \emph{labeled} \citep{papadimitriou2010web,koutra2013deltacon,soundarajan2014guide}.  Beyond the chemical distance, classic examples of distances between unlabeled graphs are the \emph{edit distance} \citep{garey2002computers,sanfeliu1983distance} and the \emph{maximum common subgraph distance} \citep{bunke1998graph,bunke1997relation}, both of which also have versions for labeled graphs. Both are metrics and are hard to compute, while existing heuristics \citep{riesen2009approximate,fankhauser2011speeding} are not metrics. 
The  \emph{reaction distance}~\citep{koca2012synthon} is also a metric  directly related to the chemical distance \citep{kvasnivcka1991reaction} when edits are restricted to edge additions and deletions. Jain \cite{jain2016geometry} also considers an extension of the chemical distance, limited to the Frobenius norm, that incorporates edge attributes. However, it is not immediately clear how to relax the above metrics \cite{jain2016geometry,koca2012synthon} to attain tractability. 

A metric  can also be induced by embedding graphs in a metric space and measuring the distance of these embeddings \citep{riesen2007graph,ferrer2010generalized,riesen2010graph}. Several works follow such an approach, mapping graphs, e.g., to spaces determined by their spectral decomposition \citep{zhu2005study,wilson2008study,elghawalby2008measuring}.  In general, in contrast to our metrics, such approaches are not as discriminative, as embeddings summarize  graph structure. Continuous relaxations of graph isomorphism, both convex and non-convex \citep{lyzinski2016graph,aflalo2015convex,umeyama1988eigendecomposition}, have found applications in a variety of contexts, including social networks \citep{koutra2013big}, computer vision \citep{schellewald2001evaluation}, shape detection \citep{sebastian2004recognition,he2006object}, and neuroscience \citep{vogelstein2011large}. None of the above works focus on metric properties of resulting relaxations, which several fail to satisfy \citep{vogelstein2011large,koutra2013big,sebastian2004recognition,he2006object}.
 
Metrics naturally arise in  data mining tasks, including clustering \citep{xing2002distance,hartigan1975clustering}, NN search \citep{clarkson2006nearest,clarkson1999nearest,beygelzimer2006cover}, and outlier detection  \citep{angiulli2002fast}. Some of these tasks become tractable or admit formal guarantees precisely when performed over a metric space. For example, finding the nearest neighbor  \citep{clarkson2006nearest,clarkson1999nearest,beygelzimer2006cover} or the diameter of a dataset \citep{indyk1999sublinear} become polylogarithimic under metric assumptions; similarly, approximation algorithms for clustering (which is NP-hard) rely on metric assumptions, whose absence leads to a deterioration on known bounds \citep{ackermann2010clustering}. Our search for metrics is motivated by these considerations.

\section{Notation and Preliminaries}
\label{sec:technical}
\noindent\textbf{Graphs.} We represent an undirected graph $G(V,E)$ with node set $V=[n] \equiv\{1,\ldots,n\} $ and  edge set $E\subseteq [n]\times[n]$  by its \emph{adjacency matrix}, i.e. $A=[a_{i,j}]_{i,j \in [n]} \in \{0,1\}^{n \times n}$ s.t. $a_{ij} = a_{ji} = 1$ if and only if $(i,j)\in E.$  
In particular, $A$ is symmetric, i.e. $A=A^\top$. We denote the set of all real, symmetric matrices by
$\Symmetric^n$.
\emph{Directed} graphs are represented by (possibly non-symmetric) binary matrices $A\in \{0,1\}^{n\times n}$, and \emph{weighted} graphs by real matrices $A\in\reals^{n\times n}$.

\noindent\textbf{Matrix Norms.} Given a matrix $A=[a_{ij}]_{i,j\in [n]}\in \reals^{n\times n}$ and a $p\in \naturals_+\cup\{\infty\}$, its \emph{induced} or \emph{operator $p$-norm}  is defined in terms of the vector $p$-norm through
$ \|A\|_p = \sup_{x\in\reals^n:\|x\|_p=1}{\|A x\|_p}, $
while its \emph{entry-wise $p$-norm} is given by
$ \|A\|_p = (\sum_{i=1}^n\sum_{j=1}^n|a_{ij}|^{p})^{1/p}, $
for $p\in \naturals_+$, and $\|A\|_\infty=\max_{i,j}|a_{i,j}|$. 
We denote the entry-wise $2$-norm (i.e., the  \emph{Frobenius} norm) as $\|\cdot\|_F$.

\noindent\textbf{Permutation, Doubly Stochastic, and Orthogonal Matrices.} We denote the set of \emph{permutation} matrices
 as
$\Permutations^n = \{P\in \{0,1\}^{n\times n}: P \onevec = \onevec, P^{\top} \onevec = \onevec\},$ 
the set of \emph{doubly-stochastic} matrices (i.e.,~the \emph{Birkhoff polytope}) as
$\DStoch^n = \{W\in [0,1]^{n\times n}: W \onevec = \onevec, W^{\top} \onevec = \onevec\}, $ 
and the set of \emph{orthogonal matrices} (i.e.,~the \emph{Stiefel manifold}) as
$\Ortho^n = \{U\in \reals^{n\times n}: UU^\top =U^\top U = I\}. $
Note that $\Permutations^n =\DStoch^n\cap \Ortho^n$. Moreover, the Birkoff-von Neumann Theorem \citep{birkhoff1946three} states that
$\DStoch^n=\conv(\Permutations^n),$  i.e., the Birkoff polytope 
 is the convex hull of $\Permutations^n$.

\noindent\textbf{Metrics.} Given a set $\Omega$, a function $d:\Omega\times\Omega \to \reals$ is called a \emph{metric}, and the pair $(\Omega,d)$ is called a \emph{metric space}, if  for all $x,y,z\in\Omega$:
\begin{subequations}
\begin{align}
d(x,y) &\geq 0 &  &\text{(non-negativity)} \label{non-neg}\\
d(x,y) &\!=\! 0 \text{ iff } x\!=\!y  & &\text{(pos.~definiteness)} \label{posdef}\\
d(x,y) &= d(y,x) &   &\text{(symmetry)} \label{symmetry}\\
d(x,y) &\!\leq\! d(x,z) \!+\!d(z,y) & & \text{(triangle~inequality)} \label{triangle}
\end{align}
A function $d$ is called a \emph{pseudometric} if it satisfies \eqref{non-neg}, \eqref{symmetry}, and \eqref{triangle}, but the positive definiteness  property \eqref{posdef} is replaced by the  (weaker) property:
\begin{align}
d(x,x) & = 0 \text{ for all }x\in \Omega.\label{weakdef}
\end{align}
\end{subequations}
 If $d$ is a pseudometric, then 
   $d(x,y)=0$ defines an equivalence relation $x \sim_d y$ over $\Omega$. A pseudometric is then a metric over $\Omega/\!\sim_d$, the quotient space of $\sim_d$.
A $d$ that satisfies \eqref{non-neg}, \eqref{posdef}, and
 \eqref{triangle} \emph{but not} the symmetry property \eqref{symmetry} is called a \emph{quasimetric}. If $d$ is a quasimetric, then its \emph{symmetric extension} $\bar{d}:\Omega\times\Omega\to \reals$, defined as
$\bar{d}(x,y) = d(x,y)+d(y,x),$
is a  metric over $\Omega.$

\noindent\textbf{Graph Isomorphism,  Chemical, and CKS Distance.}
Let $A,B\in \reals^{n\times n}$ be the adjacency matrices of two graphs $G_A$ and $G_B$. Then, $G_A$ and $G_B$ are \emph{isomorphic} if and only if there exists $P\in \Permutations^n$ s.t.
$P^\top A P =B$ or, equivalently, $AP=PB$. The \emph{chemical distance}, given by \eqref{chemical},
extends the latter relationship to capture distances between graphs.
Let $\|\cdot\|$ be a matrix norm in $\reals^{n\times n}$. 
 For some $\Omega \subseteq \reals^{n\times n}$, define $d_S:\Omega\times\Omega \to \reals_+$ as:
\begin{align}d_S(A,B) = \textstyle\min_{P\in S} \|AP-PB\| ,\label{minnorm}\end{align}
where $S\subset \reals^{n\times n}$ is a closed and bounded set, so that the infimum is indeed attained. 
Note that $d_S$ is the chemical distance \eqref{chemical} when $\Omega=\reals^{n\times n}$, $S=\Permutations^n$ and $\|\cdot\|=\|\cdot\|_F$. In CKS distance \cite{chartrand1998graph}, matrices $A,B$ contain pairwise path distances between any two nodes; equivalently, CKS is the chemical distance of two weighted complete graphs with path distances as edge weights.   Our main contribution is determining general conditions on $S$ and $\|\cdot\|$ under which $d_S$ is a metric over $\Omega$, for arbitrary weighted graphs, thereby including both the chemical and CKS distances as  special cases. 
For concreteness, we focus on distances between graphs of equal size. Extensions to graphs of unequal size are described \alt{in \citep{arxivthis}.}{in Appendix~\ref{app:size}.}

\section{A Family of Graph Metrics}
\label{sec:metric}
Our first result establishes that $d_{\Permutations^n}$ is a pseudometric  over \emph{all} weighted graphs when $\|\cdot\|$ is an \emph{arbitrary}  entry-wise or operator norm.
\begin{theorem}\label{thm:permutation} If $S=\Permutations^n$ and $\|\cdot\|$ is an arbitrary entry-wise or operator norm, then $d_S$ given by \eqref{minnorm} is a pseudometric over $\Omega=\reals^{n\times n}$.
\end{theorem}
Hence, $d_{\Permutations^n}$ is a pseudometric under any entry-wise or operator norm over arbitrary directed, weighted graphs.
Our second result states that this property extends to the \emph{relaxed} version of the chemical distance, in which permutations are replaced by doubly stochastic matrices.
\begin{theorem}\label{thm:dstochastic}\vspace*{-3mm}
If $S=\DStoch^n$ and $\|\cdot\|$ is an arbitrary entry-wise norm, then $d_S$ given by \eqref{minnorm} is a pseudometric over $\Omega=\Symmetric^{n\times n}$. If $\|\cdot\|$ is an arbitrary entry-wise or operator norm, then its symmetric extension 
$\bar{d}_S(A,B) = d_S(A,B) +d_S(B,A)$
is a pseudometric over $\Omega=\reals^{n\times n}$.
\end{theorem}
Hence, if $S=\DStoch^n$ and $\|\cdot\|$ is an arbitrary entry-wise norm, then \eqref{minnorm} defines a pseudometric over \emph{undirected} graphs. The symmetry property \eqref{symmetry} breaks if $\|\cdot\|$ is an operator norm or graphs are directed. In either  case, $d_S$ is  a quasimetric over the quotient space $\Omega/\!\sim_d$, and symmetry is  attained via the symmetric extension $\bar{d}_S$.

Theorem~\ref{thm:dstochastic} has significant practical implications. In contrast to $d_{\Permutations^n}$ and its extensions implied by Theorem~\ref{thm:permutation}, computing $d_{\DStoch^n}$ under any operator or entry-wise norm is \emph{tractable}~\cite{boyd2004convex}: it involves minimizing a convex function subject to linear constraints.
A more limited result extends to the  Stiefel manifold:
\begin{theorem}\label{thm:ortho}
If $S=\Ortho^n$ and $\|\cdot\|$ is either the operator or the entry-wise (i.e., Frobenius)  2-norm, then  $d_S$ given by \eqref{minnorm} is a pseudometric over $\Omega=\reals^{n\times n}$.
\end{theorem}
Though  \eqref{minnorm} is not a convex  problem when $S=\Ortho^n$, it is also tractable. 
\citet{umeyama1988eigendecomposition} shows that the optimization can be solved exactly when $\|\cdot\| = \|\cdot\|_F$ and $\Omega=\Symmetric^n$ (i.e., for undirected graphs) by performing a spectral decomposition on $A$ and $B$. We extend this result, showing that the same procedure also applies when  $\|\cdot\|$ is the operator $2$-norm (\alt{see Thm.\;7 in \cite{arxivthis}}{see Thm.\;\ref{th:OP_norm_d_S_comp} in Appendix \ref{sec:umeyama}}).
In the general case of directed graphs, \eqref{minnorm} is
 a classic example of a problem that can be solved through optimization on manifolds~\citep{absil2009optimization}. 

\noindent\textbf{Equivalence Classes.}  The equivalence of matrix norms  implies that all pseudometrics $d_S$ defined through \eqref{minnorm} for a given $S$ have the same quotient space $\Omega/\!\sim_{d_S}$: if $d_S(A,B)=0$ for one matrix norm $\|\cdot\|$ in \eqref{minnorm}, it will be so for all.
When $S=\Permutations^n$,  $\Omega/\!\sim_{d_{\Permutations^n}}$ is the quotient space defined by graph isomorphism: any two adjacency matrices $A,B\in \reals^{n\times n}$  satisfy
$d_{\Permutations^n} (A,B)= 0$ if and only if their (possibly weighted) graphs are isomorphic.
When $S=\DStoch^n$, the quotient space $\Omega/\!\sim_{d_{\DStoch^n}}$ has a connection to the Weisfeiler-Lehman (WL) algorithm \citep{weisfeiler1968reduction} described in \alt{\citep{arxivthis}}{Appendix~\ref{sec:WL}}: \citet{ramana1994fractional} show that  $d_{\DStoch^n}(A,B)=0$ if and only if $G_A$ and $G_B$ receive identical colors by the WL algorithm. If $S=\Ortho^n$ and $\Omega=\Symmetric^n$, i.e., graphs are undirected, then  $\Omega/\!\sim_{d_{\Ortho^n}}$ is determined by \emph{co-spectrality}: $d_{\Ortho^n}(A,B)=0$ if and only if $A,B$ have the same spectrum. When $\Omega=\reals^{n\times n}$,  $d_{\Ortho^n}(A,B)=0$ implies that $A,B$ are co-spectral, but co-spectral matrices $A,B$ do not necessarily satisfy $d_{\Ortho^n}(A,B)=0$.

\subsection{Proof of Theorems~\ref{thm:permutation}--\ref{thm:ortho}.}
We define several properties that play a crucial role in our proofs.
We say that a set $S\subseteq \reals^{n\times n}$ is \emph{closed under multiplication}  
if $P,P'\in S$ implies that $P\cdot P' \in S$. 
We say that  $S$ is \emph{closed under transposition} if $P\in S$ implies that $P^\top \in S$, and \emph{closed under inversion} if $P\in S$ implies that $P^{-1} \in S$.
Finally, given a matrix norm $\|\cdot\|$, we say that set $S$ is \emph{contractive} w.r.t.~ $\|\cdot\|$ if 
$\|AP\|\leq \|A\|$ and $\|PA\|\leq \|A\|,$  for all $P\in S$ and $A\in\reals^{n\times n}$. Put differently, $S$ is contractive if and only if every  $P\in S$ is a contraction w.r.t.~$\|\cdot\|.$
We rely on several lemmas, whose proofs can be found in \alt{\citep{arxivthis}.}{Appendix~\ref{app:manyproofs}.} The first three establish conditions under which \eqref{minnorm} satisfies the triangle inequality \eqref{triangle}, symmetry \eqref{symmetry}, and weak property \eqref{weakdef}, respectively:
\begin{lemma}\label{trianglelemma} Given a matrix norm $\|\cdot\|$, suppose that set $S$ is (a) contractive w.r.t.~$\|\cdot\|$, and (b) closed under multiplication. Then, for any $A,B,C\in \reals^{n\times n}$,  $d_S$ given by \eqref{minnorm} satisfies
$d_S(A,C) \leq d_S(A,B) +d_S (B,C).$
\end{lemma}

\begin{lemma}\label{symmetry2} Given a matrix norm $\|\cdot\|$, suppose that $S\subset \reals^{n\times n}$ is (a) contractive w.r.t.~$\|\cdot\|$, and (b)   closed under inversion. Then, for all $A,B\in \reals^{n\times n}$, $d_S(A,B) =d_S(B,A)$.
\end{lemma}

\begin{lemma}\label{weakprop} If $I\in S$, then $d_S(A,A)=0$ for all $A\in \reals^{n\times n}$.
\end{lemma}

Both the set of permutation matrices $\Permutations^n$ \emph{and} the Stiefel manifold $\Ortho^n$ are \emph{groups} w.r.t.~matrix multiplication: they are closed under multiplication, contain the identity $I$, and are closed under inversion. Hence, if they are also contractive w.r.t. a matrix norm $\|\cdot\|$, $d_{\Permutations^n}$ and $d_{\Ortho^n}$ defined in terms of this norm satisfy all assumptions of Lemmas~\ref{trianglelemma}--\ref{weakprop}. We therefore turn our attention to this property.
\begin{lemma} Let $\|\cdot\|$ be any operator or entry-wise norm. Then, $S=\Permutations^n$ is contractive w.r.t.~$\|\cdot\|$. \label{contractive_permutations}
\end{lemma}
Hence, Theorem~\ref{thm:permutation} follows as a direct corollary of Lemmas~\ref{trianglelemma}--\ref{contractive_permutations}. Indeed, $d_{\Permutations^n}$ is non-negative, symmetric by Lemmas~\ref{symmetry2} and \ref{contractive_permutations}, satifies the triangle inequality by Lemmas~\ref{trianglelemma} and \ref{contractive_permutations}, as well as property \eqref{weakdef} by Lemma \ref{weakprop}; hence  $d_{\Permutations^n}$ is a pseudometric over $\reals^{n\times n}$.
Our next lemma shows that the Stiefel manifold $\Ortho^n$ is contractive for 2-norms:
\begin{lemma} \label{contractive_ortho} 
Let $\|\cdot\|$ be the operator $2$-norm or the Frobenius norm. Then, $S=\Ortho^n$ is contractive w.r.t.~$\|\cdot\|$.  
\end{lemma}
Theorem~\ref{thm:ortho} follows from Lemmas~\ref{trianglelemma}--\ref{weakprop} and Lemma~\ref{contractive_ortho}, along with the the fact that $\Ortho^n$ is a group.
Note that $\Ortho^n$ is \emph{not} contractive w.r.t.~other norms, e.g., $\|\cdot\|_1$ or $\|\cdot\|_{\infty}$.  
 Lemma~\ref{contractive_permutations} along with the  Birkoff-von Neumann theorem  imply   that $\DStoch^n$ is also contractive:
\begin{lemma} \label{contractive_dstoch} Let $\|\cdot\|$ be any operator or entry-wise norm. Then, $\DStoch^n$ is contractive w.r.t.~$\|\cdot\|$. 
\end{lemma}
The Birkhoff polytope $\DStoch^n$ is \emph{not} a group, as it is not closed under inversion. Nevertheless, it is closed under transposition; in establishing (partial) symmetry of $d_{\DStoch^n}$, we leverage the following lemma: 
\begin{lemma}\label{symmetry1} Suppose that  $\| \cdot\|$ is transpose invariant, and  $S$ is closed under transposition. Then, $d_S(A,B)=d_S(B,A)$ for all $A,B\in \Symmetric^n$.
\end{lemma}
The first part of Theorem~\ref{thm:dstochastic} therefore follows from Lemmas ~\ref{trianglelemma}, \ref{weakprop}, and \ref{contractive_dstoch}, as $\DStoch^n$ is closed under transposition, contains the identity $I$, and is closed under multiplication, while all entry-wise norms are transpose invariant. Operator norms are not transpose invariant. However, if $\|\cdot\|$ is an operator norm, or $\Omega=\reals^{n\times n}$, then Lemma~\ref{contractive_dstoch} and Lemma~\ref{trianglelemma} imply that $d_{\DStoch^n}$ satisfies non-negativity \eqref{non-neg} and the triangle inequality \eqref{triangle}, while Lemma~\ref{weakprop} implies that it satisfies \eqref{weakdef}. These properties are inherited by extension $\bar{d}_{S}$, which also satisfies symmetry \eqref{symmetry}, and Theorem~\ref{thm:dstochastic} follows. \qedh

\section{Incorporating Metric Embeddings}
\label{sec:embedding}
We have seen that the chemical distance $d_{\Permutations^n}$ can be \emph{relaxed} to $d_{\DStoch^n}$ or $d_{\Ortho^n}$, gaining tractability while still maintaining the metric property. In practice, nodes in a graph often contain additional atributes that one might wish to leverage when computing distances. In this section, we show that such attributes can be seamlessly incorporated in $d_S$ either as soft or hard constraints, \emph{without violating the metric property}.

\noindent\textbf{Metric Embeddings.}
Given a graph $G_A$ of size $n$, a \emph{metric embedding} of $G_A$  is a mapping $\embed_A:[n]\to \tilde{\Omega} $ from the nodes of the graph to a metric space $(\tilde{\Omega},\tilde{d})$. That is, $\embed_A$ maps nodes of the graph to $\tilde{\Omega}$, where $\tilde{\Omega}$ is endowed with a metric $\tilde{d}$. We refer to a graph  endowed with an embedding $\psi_A$ as an \emph{embedded graph}, and denote this by $(A,\embed_A)$, where $A\in \reals^{n\times n}$ is the adjacency matrix of $G_A$.
 We list two examples:

\noindent\emph{Example 1: Node Attributes.} Consider an embedding of a graph to $(\reals^k,\|\cdot\|_2)$ in which every node $v\in V$ is mapped to a $k$-dimensional vector describing ``local'' attributes. These can be \emph{exogenous}: e.g., features extracted from a user's profile (age, binarized gender, etc.) in a social network. Alternatively, attributes may be \emph{endogenous} or \emph{structural}, extracted from the adjacency matrix $A$, e.g.,  the node's degree, the size of its $k$-hop neigborhood,  its page-rank, etc.\\
\noindent\emph{Example 2: Node Colors}. Let $\tilde{\Omega}$ be an arbitrary finite set endowed with the Kronecker delta as a metric, that is, for $s,s'\in \tilde{\Omega}$, $\tilde{d}(s,s')=0$ if $s=s'$, while $\tilde{d}(s,s')=\infty$ if $s\neq s'$.
Given a graph $G_A$, a mapping $\embed_A:[n]\to\tilde{\Omega}$ is then a metric embedding. The values of $\tilde{\Omega}$ are invariably called \emph{colors} or \emph{labels}, and a graph embedded in $\tilde{\Omega}$ is a \emph{colored} or \emph{labeled} graph. Colors can again be \emph{exogenous} or \emph{structural}: e.g., if the graph represents an organic molecule, colors can correspond to atoms, while structural colors can be, e.g.,  the output of the WL algorithm \alt{\citep{weisfeiler1968reduction}}{(see Appendix~\ref{sec:WL})} after $k$ iterations.

 As discussed below, node attributes translate to \emph{soft} constraints in metric \eqref{minnorm}, while node colors correspond to \emph{hard} constraints. The unified view through embeddings allows us to establish metric properties for both simultaneously~(c.f.~Thm.~\ref{thm:permutation_linear} and~\ref{thm:dstochastic_linear}) .

\noindent\textbf{Embedding Distance.} Consider two embedded graphs $(A,\embed_A)$, $(B,\embed_B)$ of size $n$ that are embedded \emph{in the same metric space $(\tilde{\Omega},\tilde{d})$. } For $u\in [n]$ a node in the first graph, and $v\in[n]$ a node in the second graph, the embedded distance between the two nodes is given by $\tilde{d}(\embed_A(u),\embed_B(v))$. Let
 $ D_{\psi_A,\psi_B} = [\tilde{d}(\embed_A(u),\embed_B(v))]_{u\in V ,v\in V}\in \reals_+^{n\times n} $
be the corresponding matrix of embedded distances.
 After mapping nodes to the same metric space, it is natural to seek $P\in \Permutations^n$ that preserve the  \emph{embedding distance}. This amounts to finding a $P\in\Permutations^n$ that minimizes:
 \begin{align} \textstyle\trace\left( P^\top D_{\psi_A,\psi_B} \right) =\sum_{u,v\in [n]}P_{u,v}\tilde{d}(\embed_A(u),\embed_B(v)).
\label{embeddingdistance} \end{align}
Note that, in the case of colored graphs and the Kronecker delta distance,  minimizing \eqref{embeddingdistance}  finds a $P \in \Permutations^n$ that maps nodes in $A$ nodes in $B$ of equal color.
It is not hard to verify\footnote{This follows from Thm.~\ref{thm:permutation_linear} for $A=B=0$, i.e., for distances between embedded graphs with no edges.} that 
$\min_{P\in \Permutations^n}  \trace\left( P^\top D_{\psi_A,\psi_B} \right) $
  induces a metric between graphs embedded in $(\tilde{\Omega},\tilde{d})$.
Despite the combinatorial nature of  $\Permutations^n$, \eqref{embeddingdistance} is a maximum weighted matching problem, which can be solved through, e.g., the Hungarian algorithm \citep{kuhn1955hungarian} in polynomial time in $n$. We note that this  metric is not as expressive as \eqref{minnorm}: depending on the definition of the embeddings $\embed_A$, $\embed_B$, attributes may only capture ``local'' similarities between nodes, as opposed to the ``global'' view of a mapping attained by \eqref{minnorm}.

\noindent\textbf{A Unified, Tractable Metric.}
Motivated by the above considerations, we focus on unifying the ``global'' metric \eqref{minnorm} with the ``local'' metrics induced by arbitrary graph embeddings. Proofs for the two theorems below are provided in the supplement.
Given a metric space $(\tilde{\Omega},\tilde{d})$,  let $\Psi^n_{\tilde{\Omega}} = \{ \embed: [n]\to \tilde{\Omega}\} $ 
be the set of all mappings from $[n]$ to $\tilde{\Omega}$.  Then, given two embedded graphs $(A,\psi_A), (B,\psi_B)\in \reals^{n\times n}\times \Psi^n_{\tilde{\Omega}}$, we define:
 \begin{align}
\begin{split}
d_{S}\left((A,\psi_A),(B,\psi_B)\right) = \min_{P\in S} \big[ &\|AP-PB\|+\ldots\\ &+ \trace (P^\top D_{\embed_A,\embed_B} ) \big]
\end{split}
  \label{addlocal}\end{align} 
for some compact set $S\subset \reals^{n\times n}$ and matrix norm $\|\cdot\|$. 
Our next result states that incorporating this linear term does not affect the pseudometric property of $d_S$.
\begin{theorem}
\label{thm:permutation_linear} 
If $S=\Permutations^n$ and $\|\cdot\|$ is an arbitrary entry-wise or operator norm, then $d_S$ given by \eqref{addlocal} is a pseudometric over the set of embedded graphs 
 $\Omega=\reals^{n\times n}\times \Psi_{\tilde{\Omega}}^n$.
\end{theorem}
\begin{figure*}[!t]
\centering
\begin{minipage}[c]{0.36\textwidth}
\includegraphics[width=1.21\textwidth]{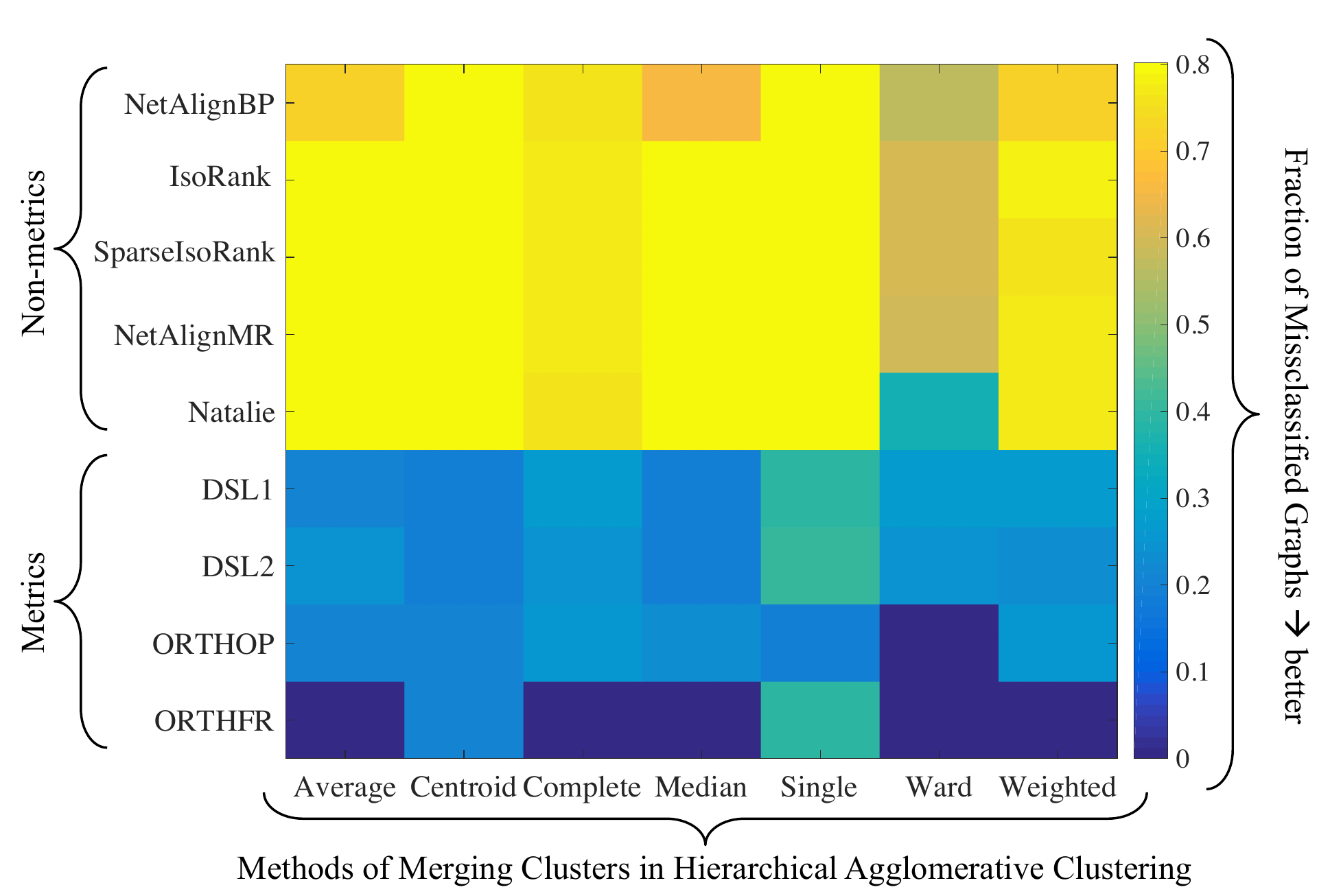}
\put(-158,134){\scriptsize (a) Clustering Misclassification Error}
\put(-67,121){\scriptsize 0.58}
\put(-67,109){\scriptsize 0.61}
\put(-67,97){\scriptsize 0.61}
\put(-67,85){\scriptsize 0.59}
\put(-67,73){\scriptsize 0.36}
\put(-104,61){\scriptsize 0.20}
\put(-104,49){\scriptsize 0.20}
\put(-67,37){\scriptsize \textcolor{white}{0.00}}
\put(-67,25){\scriptsize \textcolor{white}{0.00}}
\end{minipage}\hspace*{8mm}
\begin{minipage}[c]{0.64\textwidth}
{\includegraphics[trim = 0mm 0mm 5.5cm 0mm, clip,width=0.47\columnwidth]{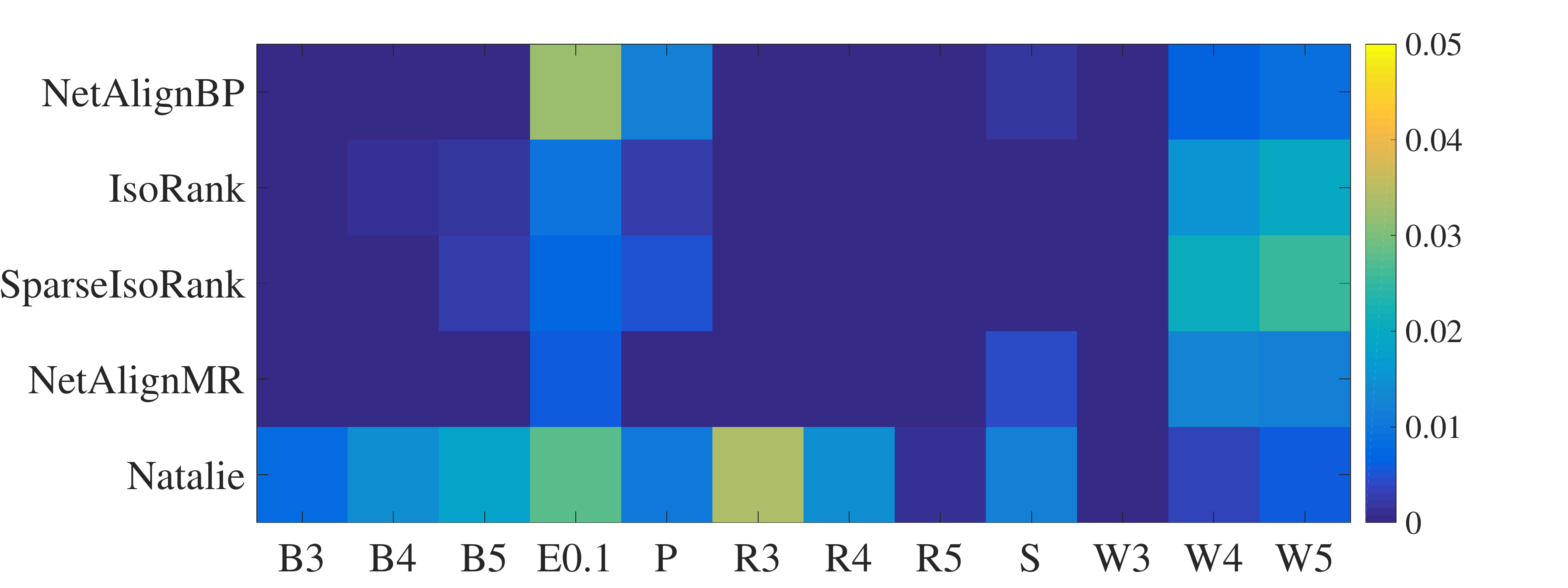}
\includegraphics[trim = 57mm 0mm 2.1cm 0mm, clip,width=0.44\columnwidth]{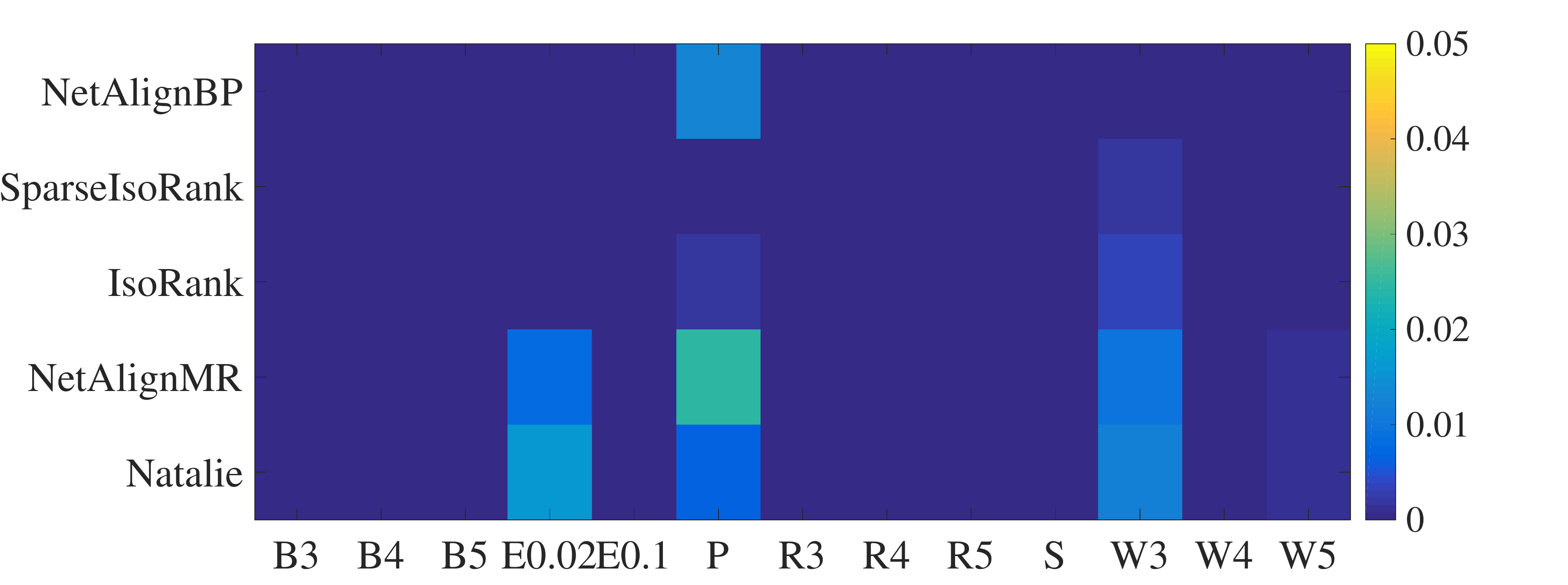}
\put(-110,65){\scriptsize (c) TIVs, $n = 50$}
\put(-230,65){\scriptsize (b) TIVs, $n = 10$}\\\medskip
\put(25,55){\tiny 
\begin{tabular}{||c l ||} 
 \hline
  & Description \\ [0.5ex] 
 \hline\hline
 B$d$ & Barabasi Albert of degree $d$ \citep{albert2002statistical} \\ 
 \hline
E$p$ &  Erd\H{o}s-R\'enyi with probability $p$ \citep{erdos1959random}  \\
 \hline
 P &  Power Law Tree \citep{mahmoud1993structure}  \\
 \hline
  R$d$ & Regular Graph of degree $d$ \citep{bollobas1998random}  \\
 \hline
 S &  Small World \citep{kleinberg2000small} \\ 
 \hline
  W$d$ & Watts Strogatz of degree $d$  \citep{watts1998collective} \\
 \hline
\end{tabular}}
\put(50,87){\scriptsize (d) Synthetic Graph Classes}
\put(170,0){
\includegraphics[trim = 9.5mm 0mm 12mm 0mm, clip,width=4.3cm]{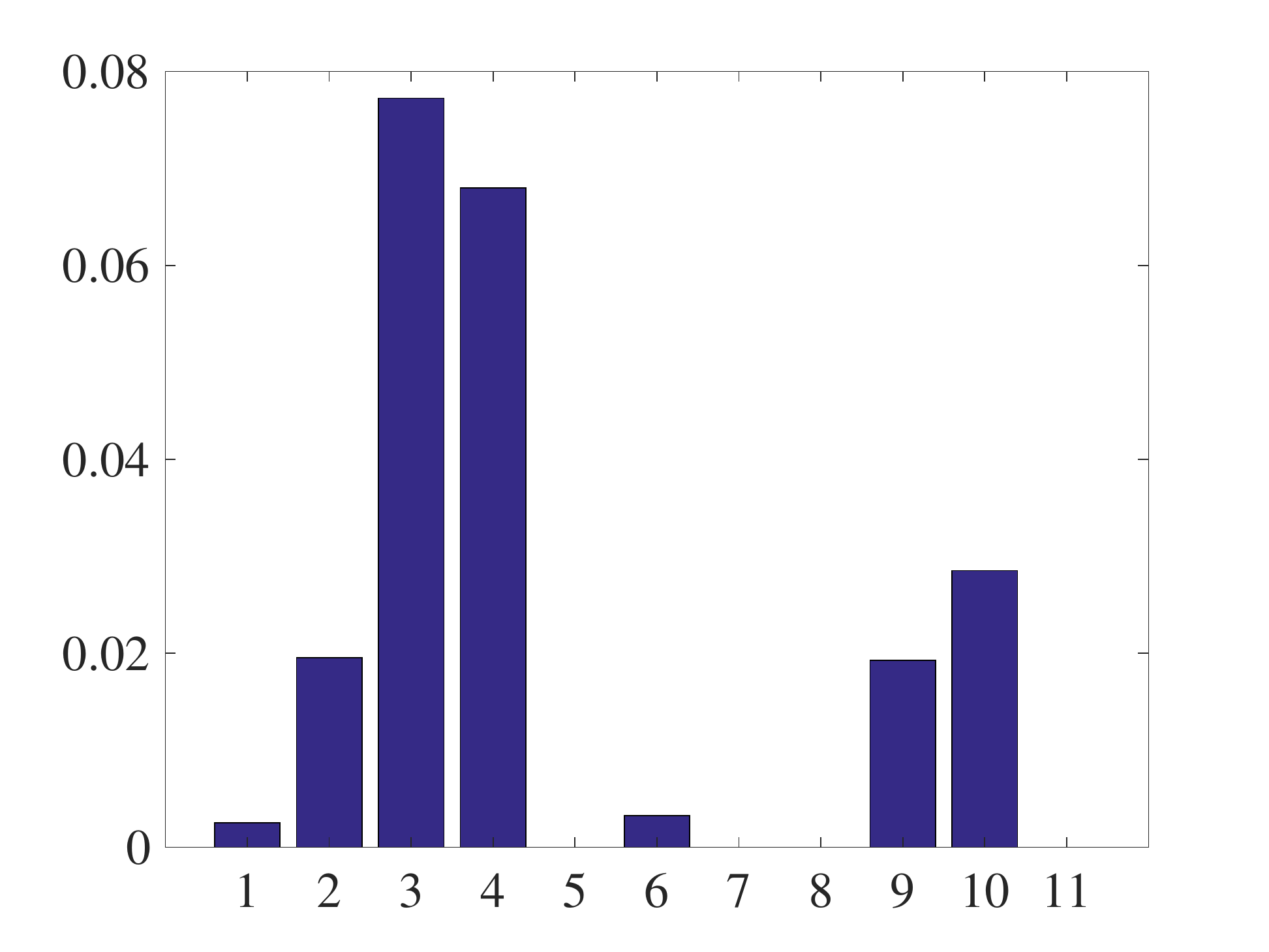}
\put(-71,89){\tiny 1 - \InnerDSLtwo}
\put(-71,84){\tiny 2 - \NetAlignBP}
\put(-71,79){\tiny 3 - \IsoRank}
\put(-71,74){\tiny 4 - \SparseIsoRank}
\put(-71,69){\tiny 5 - \NetAlignMR}
\put(-71,64){\tiny 6 - \Natalie}
\put(-71,59){\tiny 7 - \DSLone}
\put(-71,54){\tiny 8 - \DSLtwo}
\put(-71,49){\tiny 9 - \InnerPerm}
\put(-71,44){\tiny 10 - \InnerDSLone}
\put(-71,39){\tiny 11 - \EXACT}
\put(-71,34){\tiny 12 - {\bf ORTHOP}}
\put(-71,28){\tiny 13 - {\bf ORTHFR}}
\put(-95,100){\scriptsize (e) TIVs, {\em small graphs}}
} }
\end{minipage}
\vspace*{-1em}
\caption{ 
A clustering experiment using metrics and non-metrics (y-axis) for different clustering parameters (x-axis) is shown in  {\em (a), left}.  We sample graphs with $n=50$ nodes  from the six classes, shown in the adjacent table in 
{\em(d), bottom-center}. We compute distances between them using nine different algorithms from Table~\ref{table:comp}. Only the distances in our family (DSL1, DSL2, ORTHOP, and ORTHFR) are metrics. The resulting graphs are clustered using hierarchical agglomerative  clustering \citep{hartigan1975clustering} using $Average$, $Centroid$, $Complete$, $Median$, $Single$, $Ward$, $Weighted$ as a means of merging clusters. Colors represent the fraction of misclassified graphs, with the minimal misclassification rate per distance labeled explicitly. Metrics outperform other distance scores across all clustering methods. The error rate of a random guess is $\approx 0.8$. Subfigures {\em (b)}  and {\em (c), top center and right,} shows that non-metric distances  produce triangle inequality violations (TIVs) which contribute to poor clustering results; the figure shows the fraction of TIVs within different $10$-node and $50$ node graph families under these algorithms. Finally, subfigure {\em (e), bottom right},  shows the fraction of triangle inequality violations for different algorithms on the {\em small graphs} dataset of all 7-node graphs.}\label{fig:clustering_exp}
\end{figure*}
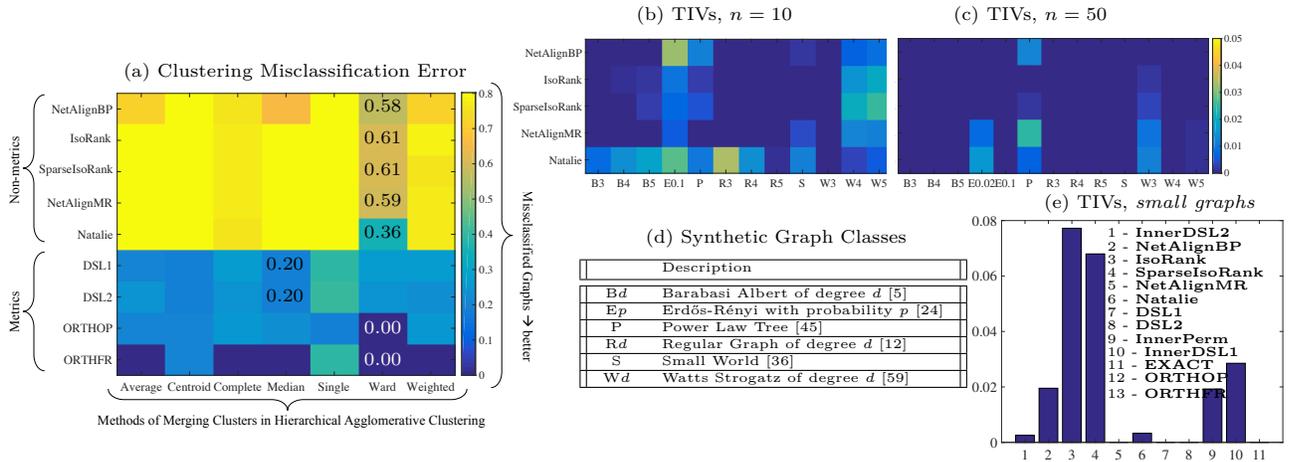

 \noindent We stress here that this result is non-obvious: is not true that adding \emph{any} linear term to $d_S$ leads to a quantity that satisfies the triangle inequality. It is precisely because $D_{\embed_A,\embed_B}$  contains pairwise distances that Theorem~\ref{thm:permutation_linear} holds.
We can similarly extend Theorem~\ref{thm:dstochastic}: 
\begin{theorem}\label{thm:dstochastic_linear} 
If $S=\DStoch^n$ and $\|\cdot\|$ is an arbitrary entry-wise  norm, then $d_S$ given by \eqref{addlocal}  is a pseudometric over
 $\Omega=\Symmetric^n\times \Psi_{\tilde{\Omega}}^n$,  the set of symmetric graphs embedded in $(\tilde{\Omega},\tilde{d})$.
 Moreover, if $\|\cdot\|$ is an arbitrary entry-wise or operator norm, then the symmetric extension
$\bar{d}_S$ of \eqref{addlocal}
is a pseudometric over  $\Omega=\reals^{n\times n}\times \Psi_{\tilde{\Omega}}^n$.
\end{theorem}
Adding the linear term \eqref{embeddingdistance} in $d_S$ has significant practical advantages. Beyond expressing exogenous attributes, a linear term involving colors, combined with a Kronecker distance, translates into \emph{hard} constraints: any permutation attaning a finite objective value \emph{must} map nodes in one graph to nodes of the same color. Theorem~\ref{thm:dstochastic_linear} therefore implies that such constraints can thus be added to the optimization problem, while maintaining the metric property.   
In practice, as the number of variables in optimization problem \eqref{minnorm} is $n^2$, incorporating such hard constraints can significantly reduce the problem's computation time; we illustrate this in the next section.  Note that  adding \eqref{embeddingdistance} to $d_{\Ortho^n}$ does \emph{not} preserve the metric propery.

\section{Experiments}
\label{sec:experiments}

\begin{figure*}[!t]
\begin{center}
\includegraphics[trim = 0mm 0mm 0cm 0mm, clip,width=0.66\columnwidth]{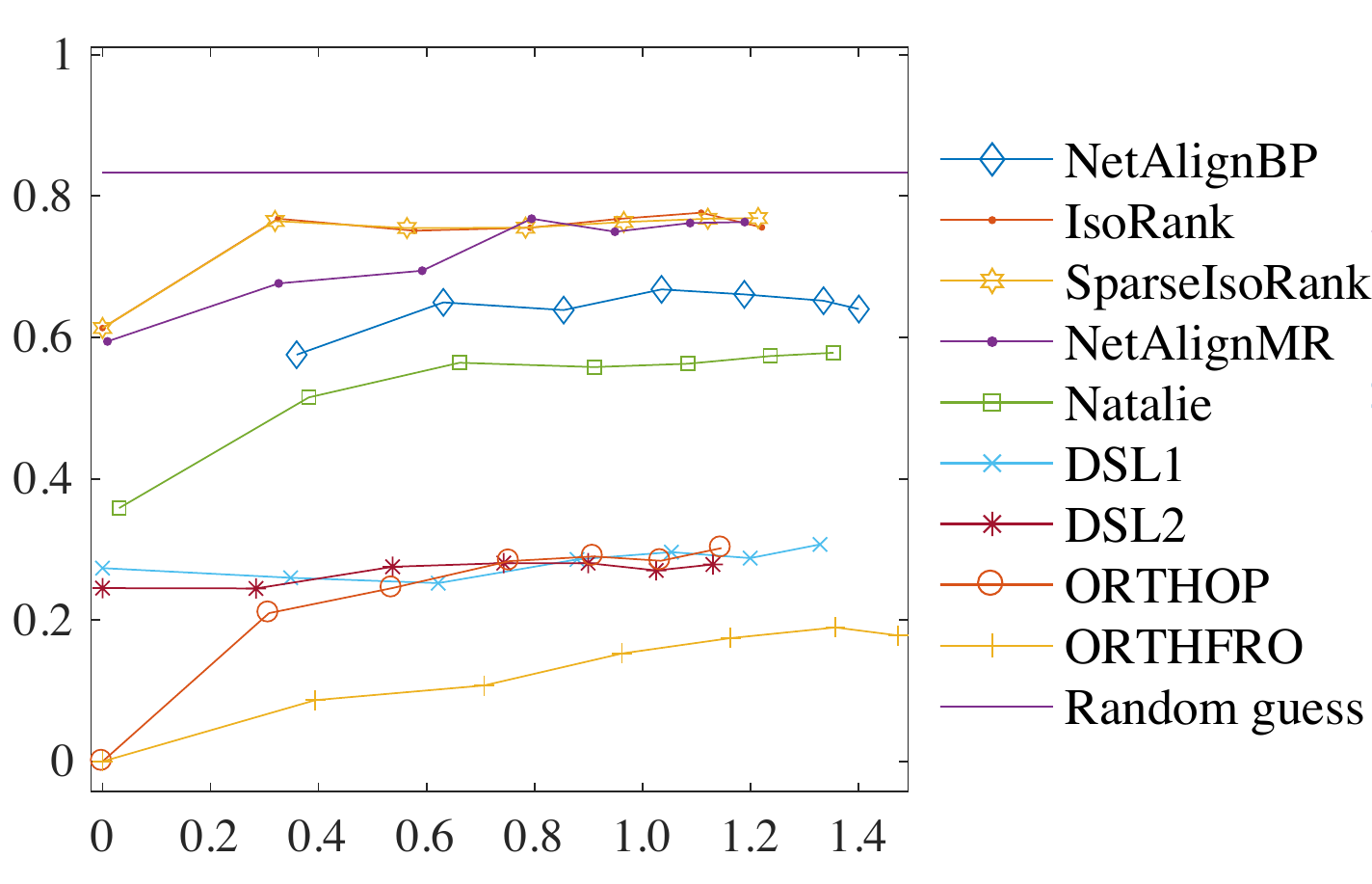}
\put(-53,08){\tiny \color{black} $ \times 10^{-2}$}
\put(-125,-03){\tiny \color{black}  Fraction of TIVs}
\put(-165,5){\tiny \color{black}  \rotatebox{90}{Fraction of Misclassified Graphs}}
\put(-125,100){\tiny (a) Effect of TIVs}
\includegraphics[trim = 5cm 9.3cm 5cm 9cm, clip,width=4.4cm]{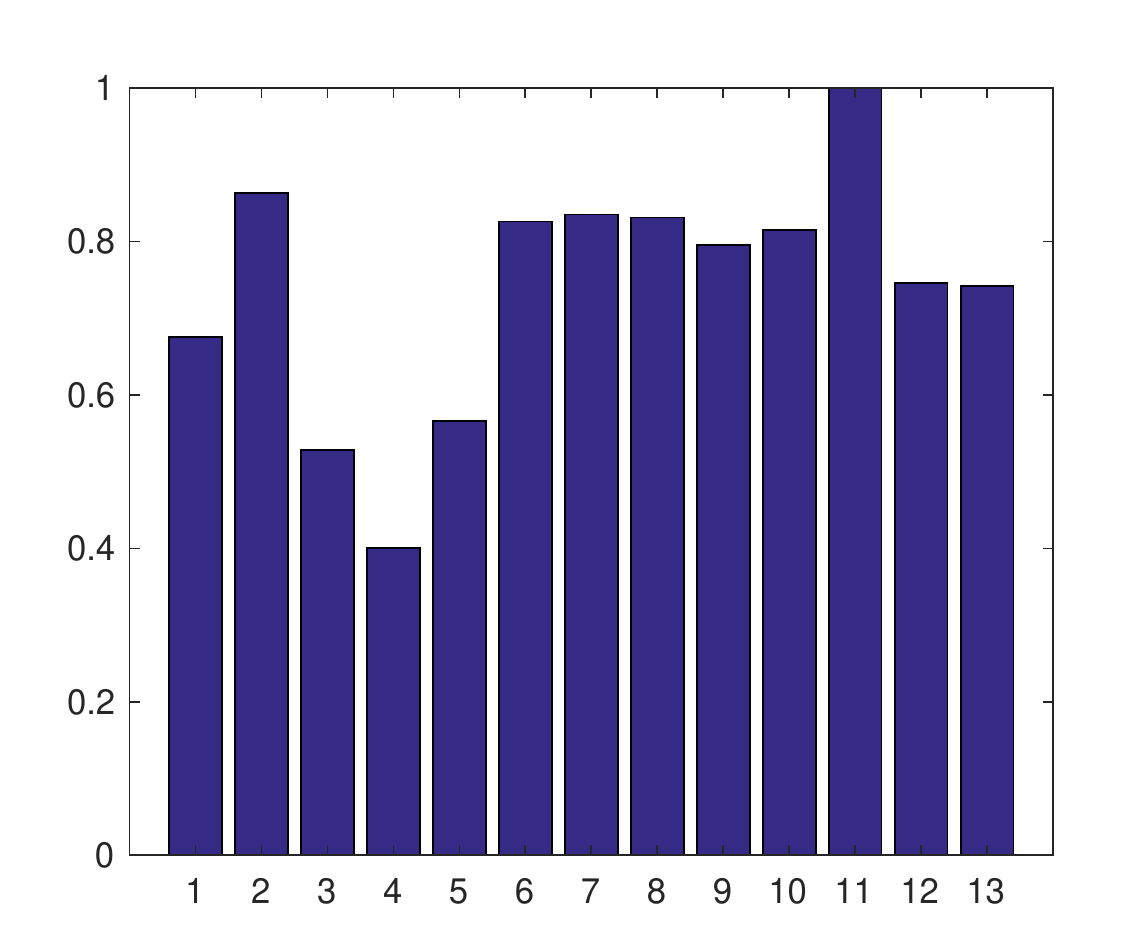}
\put(-110,100){\tiny (b) Cosine Similarity to \EXACT}
\includegraphics[trim = 9.5mm 0mm 12mm 0mm, clip,width=4.3cm]{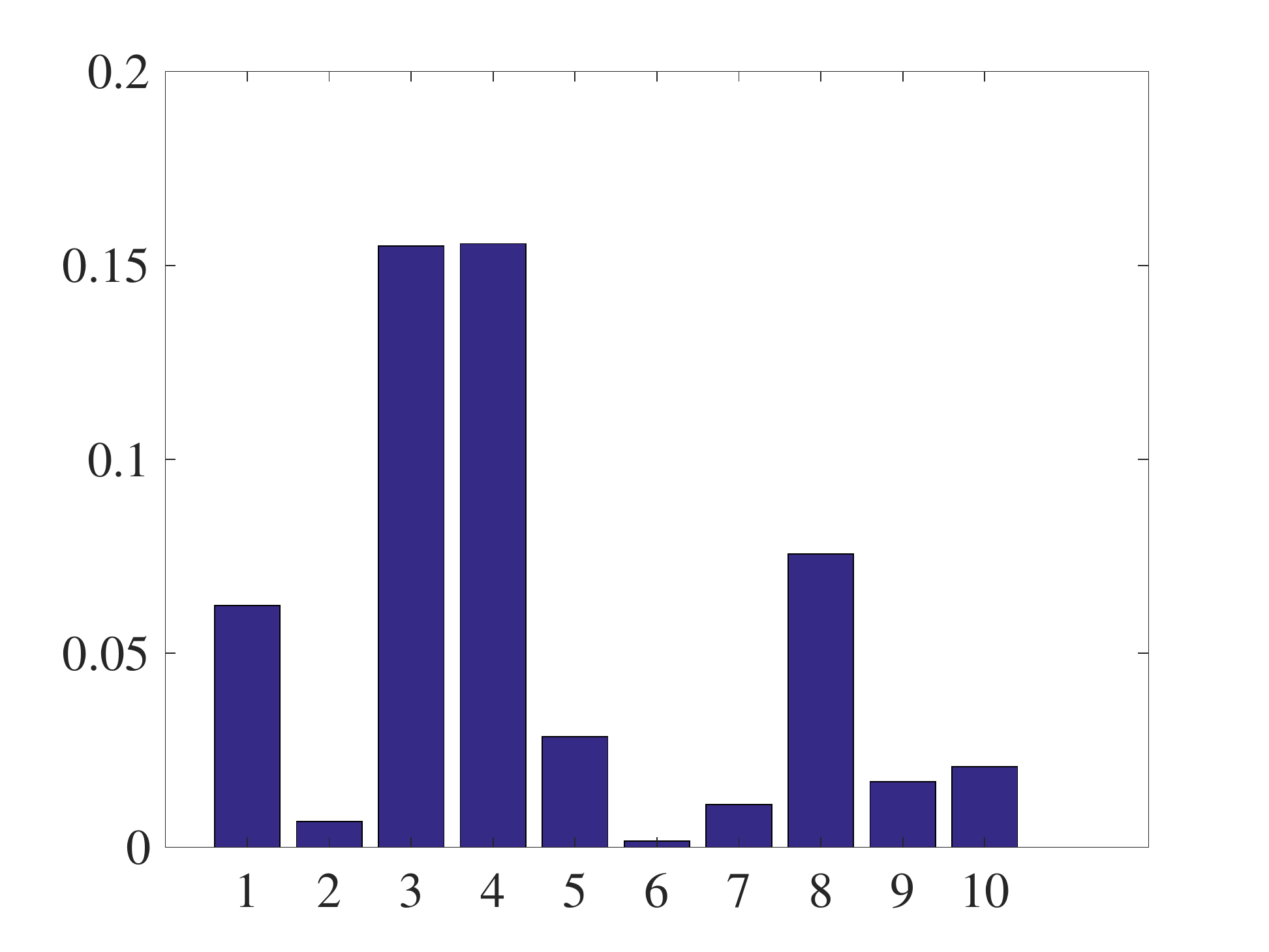}
\put(-120,100){\tiny (c) NN Graph vs.~NN Graph of \EXACT}
\put(-0,89){\tiny 1 - \InnerDSLtwo}
\put(-0,84){\tiny 2 - \NetAlignBP}
\put(-0,79){\tiny 3 - \IsoRank}
\put(-0,74){\tiny 4 - \SparseIsoRank}
\put(-0,69){\tiny 5 - \NetAlignMR}
\put(-0,64){\tiny 6 - \Natalie}
\put(-0,59){\tiny 7 - \DSLone}
\put(-0,54){\tiny 8 - \DSLtwo}
\put(-0,49){\tiny 9 - \InnerPerm}
\put(-0,44){\tiny 10 - \InnerDSLone}
\put(-0,39){\tiny 11 - \EXACT}
\put(-0,34){\tiny 12 - {\bf ORTHOP}}
\put(-0,28){\tiny 13 - {\bf ORTHFR}}
\hspace*{4em}
\caption{{ {\em (a)} Effect of introducing TIVs on the performance of different algorithms on the clustering experiment of Figure \ref{fig:clustering_exp}(a) when using the Ward method. {\em (b)} Cosine similarity between the Laplacian of distances produced by each algorithm and the one by
\EXACT. {\em (c)} Distance between nearest neighbor (NN) graphs induced by different algorithms and NN graph induced by \EXACT.}}\vspace*{-2em}
\label{fig:effect}
\end{center}
\end{figure*}
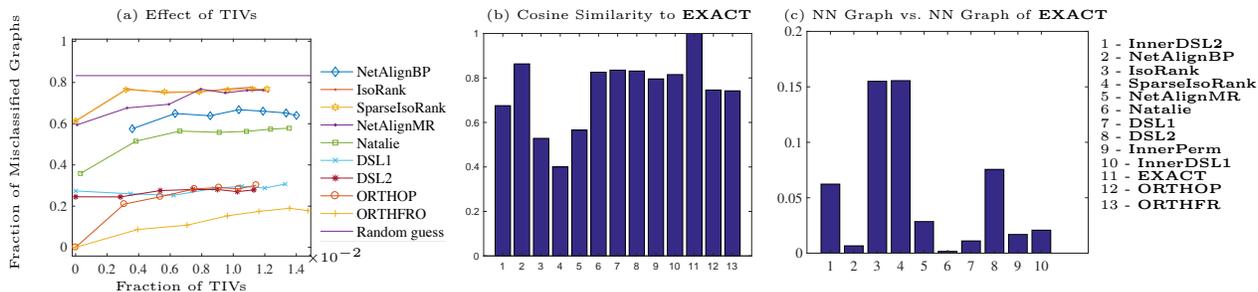

\noindent\textbf{Graphs.} We use \emph{synthetic graphs} from six classes summarized in the table in Fig.~\ref{fig:clustering_exp}(d). In addition, we use a dataset of \emph{small graphs}, comprising all $853$ connected graphs of $7$ nodes~\citep{small7graphs}.
Finally, we use a \emph{collaboration graph} with $5242$ nodes and $14496$ edges representing author collaborations
\citep{stanfordgraphs}.

\begin{table}[!t]
{\tiny
\centering
 \begin{tabular}{||c p{0.69\columnwidth}||} 
 \hline
\multicolumn{2}{||c||}{(Non-metric) Distance Score Algorithms} \\
 \hline\hline
  NetAlignBP & Network Alignment using Belief Propagation \cite{bayati2009algorithms,bayatticode} \\ 
 \hline
  IsoRank & Neighborhood Topology Isomorphism using Page Rank \cite{singh2007pairwise,bayatticode} \\
 \hline
 SparseIsoRank &Neighborhood Topology Sparse Isomorphism using Page Rank \cite{bayati2009algorithms,bayatticode} \\
 \hline
 InnerPerm & Inner Product Matching with Permutations \cite{lyzinski2016graph}  \\
 \hline
 InnerDSL1 & Inner Product Matching with  Matrices in $\DStoch^n$ and entry-wise 1-norm \cite{lyzinski2016graph}  \\
 \hline
InnerDSL2 &  Inner Product Matching with  Matrices in $\DStoch^n$ and Frobenius norm \cite{lyzinski2016graph}  \\
 \hline
 NetAlignMR & Iterative Matching Relaxation \cite{klau2009new,bayatticode}  \\
 \hline
Natalie (V2.0) & Improved Iterative Matching Relaxation  \cite{el2015natalie,nataliecode}  \\ 
 \hline
 \hline
\multicolumn{2}{||c||}{Metrics from our Family \eqref{minnorm}} \\
\hline
\hline
EXACT & Chemical Distance via brute force search over GPU  \\
 \hline
 DSL1&  Doubly Stochastic Chemical Distance $d_{\DStoch^n}$ with entry-wise 1-norm  \\
 \hline
 DSL2& Doubly Stochastic Chemical Distance $d_{\DStoch^n}$ with Frobenius norm  \\
 \hline
 ORTHOP&  Orthogonal Relaxation of Chemical Distance $d_{\Ortho^n}$ with operator 2-norm  \\
 \hline
 ORTHFR& Orthogonal Relaxation of Chemical Distance $d_{\Ortho^n}$ with Frobenius norm   \\
 \hline
\end{tabular}
}
\caption{Competitor Distance Scores \& Our Metrics}\label{table:comp}

\end{table}

\begin{figure}[!t]
\begin{minipage}[c]{\columnwidth}
{\centering
{
\newcolumntype{C}[1]{>{\!\centering}m{#1}}
\tiny
\begin{tabular}{||C{0.1cm} C{0.6cm} C{1.2cm} c||}
\hline
$k$& $\|P\|_0$ & $\|\!AP\!\!-\!\!PA\!\|_0$ & $\tau$\\
\hline
1 & 3,747,960 & 100.569 & 133s\\
 \hline
2 & 239,048  & 3,004 & 104s\\
 \hline
3 & 182,474  &2,036& 136s\\
 \hline
4  &  182,016  &2,030& 169s\\
 \hline
5  & 182,006  &2,030& 200s\\
\hline
\end{tabular}
\put(-120,27){\footnotesize (a) Coloring Constraints}\vspace*{3em}
}{
\includegraphics[ trim = 3mm 70mm 30mm 30mm,  width=0.40\columnwidth]{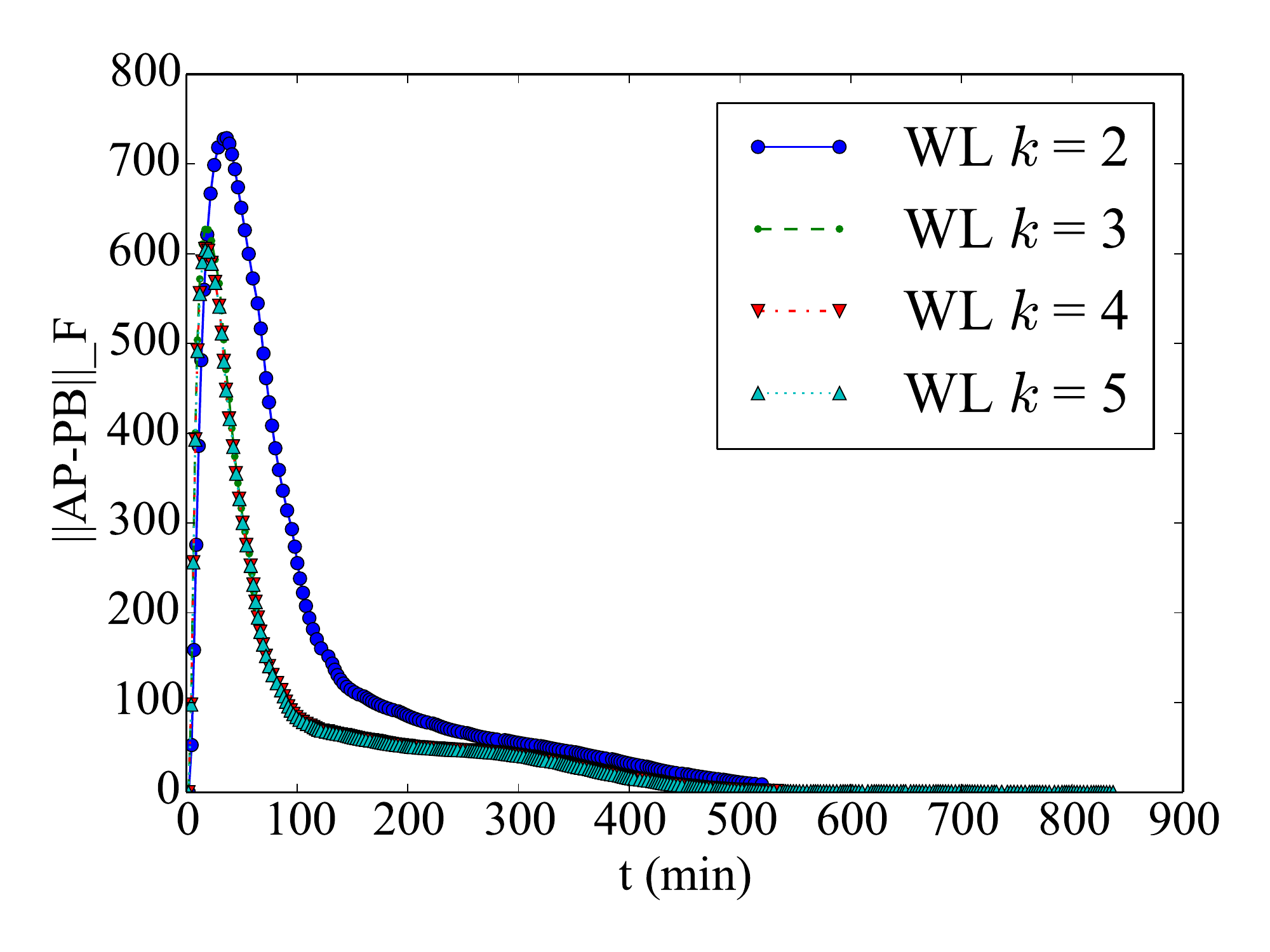}
\put(-90,45){\footnotesize (b) Convergence of ADMM}
}\vspace*{-1em}
}
\end{minipage}
\caption{{\em (a)}  Effect of coloring/hard constraints on the numbers of variables ($\|P\|_0$) and terms of objective ($\|AP-PA\|_0$) using $k$ iterations of the WL coloring algorithm. The last column shows the execution time of WL on a 40 CPU machine using Apache Spark  \citep{zaharia2010spark}.  {\em (b)} Convergence of ADMM algorithm \cite{boyd2011distributed} computing \DSLtwo on two copies of the collaboration graph as a function of time, implemented using Apache Spark  \citep{zaharia2010spark} on a 40 CPU machine.}\vspace*{-1em} \label{table:admm}
\end{figure}

\noindent\textbf{Algorithms.} We compare our metrics to several competitors outlined in Table~\ref{table:comp} \alt{(see also \citep{arxivthis} for additional details).}{(see also Appendix~\ref{app:alg}).}
 All receive only two unlabeled undirected simple graphs $A$ and $B$ and output a matching a matrix $\hat{P}$ either in
$\DStoch^n$ or in $\Permutations^n$ estimating $P^{*}$. If $\hat{P} \in \Permutations^n$,   we compute $\|A\hat{P} - \hat{P}B\|_1$.
If $\hat{P} \in \DStoch^n$, then we compute both $\|A\hat{P} -\hat{P}B\|_1$ and $\|A\hat{P} - \hat{P}B\|_F$; all norms are entry-wise. We also implement our two relaxations $d_{\DStoch}$ and $d_{\Ortho^n}$, for two different matrix norm combinations.

\noindent\textbf{Clustering Graphs.} 
The difference between our metrics and non-metrics is  striking when clustering graphs. This is illustrated by the clustering experiment shown in Fig.~\ref{fig:clustering_exp}(a). Graphs of size $n=50$ from the 6 classes in Fig.~\ref{fig:clustering_exp}(d) are clustered together through hierarchical agglomerative clustering. We compute distances between them using nine different algorithms; only the distances in our family (DSL1, DSL2, ORTHOP, and ORTHFR) are metrics. The quality of clusters induced by our metrics are far superior than clusters induced by non-metrics; in fact, \textbf{ORTHOP} and \textbf{ORTHFR} can lead to no misclassifications. This experiment strongly suggests our produced metrics correctly capture the topology of the metric space between these larger graphs.

\noindent\textbf{Triangle Inequality Violations (TIV).}
Given graphs $A$, $B$ and $C$ and a distance $d$, a TIV occurs when $d(A,C) > d(A,B) +d(B,C)$. Being metrics, none of our distances induce TIVs; this is not the case for the remaining algorithms in Table~\ref{table:comp}.  Fig.~\ref{fig:clustering_exp}(b) and (c)  show the TIV fraction across the synthetic graphs of Fig.~\ref{fig:clustering_exp}(d), while Fig.~\ref{fig:clustering_exp}(e) shows the fraction of TIVs found on the $853$ small graphs ($n=7$).   
 \NetAlignMR also produces no TIVs on the small graphs, but it does induce TIVs in synthetic graphs. We observe that it is easier to find TIVs when graphs are close:  in synthetic graphs, TIVs abound for $n=10$. No algorithm performs well across all categories of graphs.

\noindent\textbf{Effect of TIVs on Clustering.} Next, to investigate the effect of TIVs on clustering, we artificially introduced triangle inequality violations into the pairs of distances between graphs. We then re-evaluated clustering performance for hierarchical agglomerative clustering
using the \emph{Ward} method, which performed best in Fig.~\ref{fig:clustering_exp}(a). Fig.~\ref{fig:effect}(a) shows the 
fraction of misclassified graphs as the fraction
of TIVs introduced increases. To incur as small a perturmbation on distances as possible, we introduce TIVs as follows:  For every three graphs, $A,B,C$,
with probability $p$, we set $d(A,C)=d(A,B)+d(B,C)$. Although this does not introduce a TIV w.r.t.~$A$,$B$, and $C$, this distortion does introduce TIVs w.r.t. other triplets involving $A$ and $C$. We repeat this 20 times for each algorithm and each value of $p$, and compute the average fraction of TIVs, shown in the $x$-axis,
and the average fraction of misclassified graphs, shown in the $y$-axis. As little as $1\%$ TIVs significantly deteriorate clustering performance.
We also see that, even after introducing TIVs, clustering based on metrics
outperforms clustering based on non-metrics.

\noindent\textbf{Comparison to Chemical Distance.} 
We compare how different distance scores relate to the chemical distance \EXACT through two experiments on the small graphs (computation on larger graphs is prohibitive).
In Figure \ref{fig:effect}(b), we compare the distances between small graphs with $7$ nodes produced by the different algorithms and \EXACT using the DISTATIS method of \cite{abdi2005distatis}. Let $D\in \reals_+^{835\times 835}$ be the matrix of distances between graphs under an algorithm. DISTATIS computes the normalized Laplacian of this matrix, given by $L=-UDU/\|UDU\|_2$  where $U = I - \frac{{\bf 1}{\bf 1}^\top}{n}$. The DISTATIS score is the cosine similarity of such Laplacians (vectorized). We see that our metrics produce distances attaining high similarity with \EXACT, though   \NetAlignBP has the highest similarity. We measure proximity to \EXACT with an additional test. Given $D$, we compute the nearest neighbor (NN) meta-graph by connecting a graph in $D$ to every graph at distance less than its average distance to other graps. This results in a (labeled) meta-graph, which we can compare to the NN meta-graph induced by other algorithms, measuring the fraction of distinct edges.  Fig.~\ref{fig:effect}(c) shows that our algorithms perform quite well, though \Natalie yields the smallest distance to \EXACT.

\noindent\textbf{Incorporating Constraints.}
Computation costs  can be reduced through metric embeddings, as in \eqref{addlocal}.
To show this, we produce a copy of the $5242$ node collaboration graph with permuted node labels.
We then run the WL algorithm \alt{\citep{weisfeiler1968reduction}}{(see Appendix~\ref{sec:WL})} to produce structural colors, which induce coloring constraints on $P\in \DStoch^n$.  
 The support of $P$ (i.e., the number of variables in the optimization \eqref{minnorm}), the support of $AP-PA$ (i.e., the number of non-zero summation terms in the objective of \eqref{minnorm}), as well as the execution time $\tau$ of the WL algorithm, are summarized in  Fig.~\ref{table:admm}(b). 
The original unconstrained problem involves $5242^2\approx 27.4$M variables. However, after using WL and induced costraints, the effective dimension of
the optimization problem \eqref{minnorm} reduces considerably.
This, in turn, speeds up convergence time, shown in Fig.~\ref{table:admm}(b): including the time to compute constraints, a solution is found 110 times faster after the introduction of the constraints.

\vspace{-0.4cm}
\section{Conclusion}
\label{sec:conclusion}

Our work suggests that incorporating soft and hard constraints has a great potential to further improve the efficiency of our metrics. In future work, we intend to investigate and characterize the resulting equivalence classes under different soft and hard constraints and to quantify these gains in efficiency, especially in parallel implementations like ADMM. Determining the necessity of the conditions used in proving that $d_S$ is a metric is also an open problem. 
 
\section*{Acknowledgements}

The authors gratefully acknowledge the support of the National Science Foundation (grants IIS-1741197,IIS-1741129) and of the National Institutes of Health (grant 1U01AI124302).   

\begin{footnotesize}

\bibliographystyle{plainnat}
\end{footnotesize}

\newpage
\appendix
\section{Proof of Lemmas~\ref{trianglelemma}--\ref{symmetry1}}\label{app:manyproofs}
\subsection{Proof of Lemma~\ref{trianglelemma}}\label{sec:proofoflemmatrianglelemma}
Consider $P' \in \argmin_{P\in S} \|AP-PB \|$, and $P'' \in \argmin_{P\in S} \|BP-PC\|$. Then, from closure under multiplication, $P'P''\in S$. Hence,
\begin{align*}
d_S(A,C) &\leq \| AP'P'' - P'P'' C\| \\ &\leq \| AP'P'' - P'BP'' \|+ \|P'BP''- P'P'' C\|\\
&=\|(A P' - P' B) P''\|  + \|P' (BP''- P''C) \| \\
&\leq \|A P' - P' B\| +  \|BP''- P''C\| 
\end{align*}
where the last inequality follows from the fact that $P',P''$ are contractions.\qedh

\subsection{Proof of Lemma~\ref{symmetry2}}\label{sec:proofoflemmasymmetry2}
Observe that property (b)  implies that, for all $P\in S$, $P$ is invertible and $P^{-1}\in S$. Hence,
$\|AP \!-\!PB\| \!=\!\| P (P^{-1}\!A\! -\!BP^{-1})P\| \!\leq\!  \|BP^{-1} \!-\!P^{-1} A \|,$  as $P$ is a contraction w.r.t~$\|\cdot\|$. We can similarly show that
$\|BP^{-1} -P^{-1} A \| \leq \|AP-PB\|,$
hence $\|AP-PB\|=\|BP^{-1} -P^{-1}A\|.$ As $S$ is closed under inversion, $\min_{P\in S} f(P)=\min_{P:P^{-1}\in S}f(P)$, so 
$d_S(A,B) = \min_{P\in S} \| BP^{-1} -P^{-1}A\| = \min_{P:P^{-1} \in S}  \| BP^{-1} -P^{-1}A\|=  \min_{P\in S} \| BP -P A\|  = d_S(B,A).$\qedh

\subsection{Proof of Lemma~\ref{weakprop}}\label{sec:proofoflemmaweakprop}
If $I\in S$, then $0\leq d_S(A,A)\leq \|AI-IA\|=0$. \qedh

\subsection{Proof of Lemma~\ref{contractive_permutations}}\label{sec:proofoflemmacontractivepermutations}
Observe first that all vector $p$-norms are invariant to permutations a vector's entries; hence, for any vector $x\in\reals^d$, if $P\in \Permutations^n$, $\|Px\|_p=\|x\|_p$.
Hence, if $\|\cdot\|$ is an operator $p$-norm,
$\|P\|=1,$ \text{for all}~$P\in S.$
Every operator norm is  \emph{submultiplicative}; as a result
$\|P A\|\leq \|P\| \|A\|=\|A\|$ and, similarly,  $\| A P\|\leq\|A\|$, 
so  the lemma follows for operator norms. On the other hand, if $\| \cdot\|$ is an entry-wise norm, then $\|A\|$  is invariant to permutations of either $A$'s rows or columns. Matrices $PA$ and $AP$ precisely amount to such permutations, so
$\|PA\|=\|AP\|=\|A\|$
and the lemma follows also for entrywise norms.\qedh

\subsection{Proof of Lemma~\ref{contractive_ortho}}\label{sec:proofoflemmacontractive_ortho}
Any $U\in \Ortho^n$ is an orthogonal matrix; hence, $\|U\|_2=\|U\|_F=1$. Both norms are submultiplicative: the first as an operator norm, the second from the Cauchy-Schwartz inequality. Hence, for $U\in \Ortho^n$, we have
$\|UA\|\leq \|U\| \|A\| = \|A\|.$\qedh

\subsection{Proof of Lemma~\ref{contractive_dstoch}}\label{sec:proofoflemma}
By the Birkoff-con Neumann theorem \citep{birkhoff1946three}, $\DStoch^n=\conv(\Permutations^n)$. Hence, for any $W\in \DStoch^n$ there exist $P_i\in \Permutations^n$, $\theta_i>0$, $i=1,\ldots,k$, such that
$W = \sum_{i =1}^k \theta_i P_i$ and $\sum_{i=1}^k\theta_i =1.$
 Both operator and entrywise $p$-norms are convex functions; hence, by Jensen's inequality, for any $A\in \reals^{n\times N}$:
$\| W A\| \leq\textstyle \sum_{i=1}^k\theta_i \|P_iA\| \leq  \textstyle\sum_{i=1}^k\theta_i\|A\| =\|A\|$
where the last ineqality follows  by Lemma~\ref{contractive_permutations}. The statement $\|AW\|\leq \|A\|$ follows similarly.\qedh

\subsection{Proof of Lemma~\ref{symmetry1}}\label{sec:proofoflemmasymmetry1}
By transpose invariance and the symmetry of $A$ and $B$, we have that:
$\|AP-PB\| = \|B P^\top-P^\top A\|.$ Moreover, as $S$ is closed under transposition, $\min_{P\in S}f(P)=\min_{P^\top \in S} f(P)$. Hence, 
$d_S(A,B) = \min_{P\in S} \|BP^\top -P^\top A\| = \min_{P^\top \in S}  \|BP^\top -P^\top A\| = d_S(B,A).$ \qedh

\section{Proof of Theorems~\ref{thm:permutation_linear} and~\ref{thm:dstochastic_linear}}

We begin by establishing conditions under which $d_S$ satisfies the triangle inequality \eqref{triangle}. We note that, in contrast to Lemma~\ref{trianglelemma}, we require the additional condition that $S\subseteq \DStoch^n$, which is not satisfied by $\Ortho^n$.
\begin{lemma}\label{newtrianglelemma} Given a norm $\|\cdot\|$, suppose that $S$ is (a) contractive w.r.t.~$\|\cdot\|$,
(b)  closed under multiplication, and   (c) is a subset of $\DStoch^n$, i.e., contains only doubly stochastic matrices. Then, for any $(A,\psi_A),(B,\psi_B),(C,\embed_C)$  in  $\reals^{n\times n} \times \Psi_{\tilde{\Omega}}$,
$d_S((A,\embed_A),(C,\embed_B)) \leq d_S((A,\embed_A),(B,\embed_B)) +d_S ((B,\embed_B),(C,\embed_C)).$
\end{lemma}
\begin{proof}
Consider $$P' \in \argmin_{P\in S} \left(\|AP-PB \|+\trace \left(P^\top D_{\embed_A,\embed_B}\right)\right), $$
 and $$P'' \in \argmin_{P\in S} \left(\|BP-PC\| +\trace \left( P^\top D_{\embed_B,\embed_C}\right)\right).$$ Then, from closure under multiplication, $P'P''\in S$. 
We have that
\begin{align*}
\begin{split}d_S((A,\embed_A),(C,\embed_C)) \leq
 \| AP'P'' - P'P'' C\| \\+  \trace \left[(P'P'')^\top D_{\embed_A\embed_C}\right]\end{split}
\end{align*}
As in the proof of Lemma~\ref{trianglelemma}, we can show that
\begin{align*}
 \| AP'P'' - P'P'' C\|& \leq  \|A P' - P' B\|  + \|BP''- P''C \|
\end{align*}
using the fact that both $P'$ and $P''$ are contractions, while
\begin{align*}
\trace\big[(P'&P'')^\top D_{\embed_A\embed_C}\big] =  \displaybreak[0]\\&=\sum_{u,v\in [n]} \sum_{k\in[n]} \left( P'_{uk}P''_{kv} \tilde{d}(\embed_A(u),\embed_C(v))) \right)\displaybreak[0]\\
&\leq \sum_{u,v\in [n]} \sum_{k\in[n]}\big[  P'_{uk}P''_{kv}\big( \tilde{d}(\embed_A(u),\embed_B(k))\\&\qquad + \tilde{d}(\embed_B(k),\embed_C(v))  \big)\big]\\
& \text{ (as }\tilde{d}\text{ is a metric, and}P',P''  \text{are non-negative) }\displaybreak[0]\\
&= \sum_{u,k\in [n]} P'_{uk}\ \tilde{d}(\embed_A(u),\embed_B(k))  \sum_{v\in[n]} P''_{kv}\\ 
&\quad + \sum_{k,v\in [n]}P''_{kv} \tilde{d}(\embed_B(k),\embed_C(v))  \sum_{u\in[n]} P'_{uk}\displaybreak[0]\\
&\leq \trace \left((P')^\top D_{\embed_A,\embed_B}\right) + \trace \left((P'')^\top  D_{\embed_B,\embed_C}\right),
\end{align*}
where the last inequality follows as both $P,P^\top$ are $\|\cdot\|_1$-norm bounded by 1 for every $P\in S$. 
\end{proof}
The weak property \eqref{weakdef} is again satisfied provided the identity is included in $S$.
\begin{lemma}\label{newweakprop} If $I\in S$, then $d_S( (A,\embed_A),(A,\embed_A))=0$ for all $A\in \reals^{n\times n}$.
\end{lemma}
\begin{proof}
Indeed, $0\leq d_S((A,\embed_A,(A,\embed_A))\leq \|AI-IA\|+\sum_{u\in [n]}\tilde{d}(\embed_A(u),\embed_A(u)) =0$.
\end{proof} 
To attain symmetry over $\Omega=\reals^{n\times n}$, we again rely on closure under inversion, as in Lemma~\ref{newsymmetry2}; nonetheless, in contrast to Lemma~\ref{newsymmetry2}, due to the linear term, we also need to assume orthogonality of $S$.
\begin{lemma}\label{newsymmetry2} 
Given a norm $\|\cdot\|$, suppose that $S$ (a) is contractive w.r.t.~$\|\cdot\|$,
(b)  is  closed under inversion, and (c) is a subset of $\Ortho^n$, i.e., contains only orthogonal matrices.
Then,    $d_S((A,\embed_A),(B,\embed_B))= d_S((B,\embed_B),(A,\embed_A))  $ for all $ (A,\embed_A),(B,\embed_B) \in \reals^{n\times n} \times  \Psi_{\tilde{\Omega}} $.
\end{lemma}
\begin{proof}
As in the proof of Lemma~\ref{symmetry2}, we can show that contractiveness w.r.t.~$\|\cdot\|$ along with closure under inversion imply that:
$\|AP-PB\|=\|BP^{-1} -P^{-1}A\|.$ 
As $S$ is closed under inversion, $\min_{P\in S} f(P)=\min_{P:P^{-1}\in S}f(P)$ for all $f:S\to\reals$, while orthogonality implies $P^{-1}=P^\top$ for all $P\in S$. Hence, $d_S((A,\embed_A),(B,\embed_B))$ equals:
\begin{align*}
 \min_{P\in S}&\left[ \|AP -P B\| + \trace\left( P^\top D_{\embed_A,\embed_B} \right) \right]
\\&= \min_{P\in S} \left[ \| BP^{-1} -P^{-1}A\|  + \trace\left( P^{-1} D_{\embed_A,\embed_B}   \right) \right]     \\
&= \min_{P\in S} \left[ \| BP^{-1} -P^{-1}A\|  + \trace\left(  \left(P^{-1}\right)^{\top} D_{\embed_A,\embed_B }^\top \right)\right]   \\ 
&= \min_{P^{-1}\in S} \left[ \| BP^{-1} -P^{-1}A\|  + \trace\left(  \left(P^{-1}\right)^{\top} D_{\embed_A,\embed_B }^\top \right)\right]     \\
&= d_S((B,\embed_B),(A,\embed_A)).\qquad\qquad\qquad\qquad\qquad\qedhere
\end{align*}
\end{proof}

Theorem~\ref{thm:permutation_linear} therefore follows from the above lemmas, as $S=\Permutations^n$ contains $I$, it is closed under multiplication and inversion, is a subset of $\DStoch^n\cap \Ortho^n$, and is contractive w.r.t.~all operator and entrywise norms. Theorem~\ref{thm:dstochastic_linear} also follows by using the following lemma, along with Lemmas~\ref{newtrianglelemma} and~\ref{newweakprop}.
\begin{lemma}\label{newsymmetry1} Suppose that  $\| \cdot\|$ is transpose invariant, and  $S$ is closed under transposition. Then, $d_S((A,\embed_A),(B,\embed_B))= d_S((B,\embed_B),(A,\embed_A))  $ for all $ (A,\embed_A),(B,\embed_B) \in \Symmetric^n \times  \Psi_{\tilde{\Omega}} $.
\end{lemma}
\begin{proof}
By transpose invariance of $\|\cdot \|$ and the symmetry of $A$ and $B$, we have that:
$\|AP-PB\| = \|B P^\top-P^\top A\|.$ Moreover, as $S$ is closed under transposition, $\min_{P\in S}f(P)=\min_{P^\top \in S} f(P)$ for any $f:S\to\reals$. Hence, $d_S((A,\embed_A),(B,\embed_B))$ equals
\begin{align*}
   \min_{P\in S}&\left[ \|AP -P B\| + \trace\left( P^\top D_{\embed_A,\embed_B} \right) \right]\\
& = \min_{P\in S} \left[ \|BP^\top -P^\top A\| + \trace\left( P D_{\embed_A,\embed_B}^\top \right) \right]\\
&= \min_{P^\top \in S}  \|BP^\top -P^\top A\|  + \trace\left( (P^\top)^\top D_{\embed_B,\embed_A} \right)\\
&= d_S((B,\embed_B),(A,\embed_A)).\qquad\qquad\qquad\qquad\qquad\qedhere
 \end{align*}
\end{proof}

\section{Metric Computation Over the Stiefler Manifold.}\label{sec:umeyama}
In this section, we describe how to compute the metric $d_S$ in polynomial time when $ S = \Ortho^n$ and $\|\cdot\|$ is the Frobenious norm or the operator 2-norm. The  algorithm for the Frobenius norm, and the proof of its correctness, is due to \cite{umeyama1988eigendecomposition}; we include it in this appendix for completeness, along with its extension to the operator norm.

Both cases make use of the following lemma:
\begin{lemma}\label{eqlemma} For any matrix $M\in\reals^{n\times n}$ 
and any matrix $P\in\Ortho^n$
we have that $\|PM\| = \|MP\| = \|M\|$, where $\|\cdot\|$ is either the Frobenius or operator 2-norm.
\end{lemma}
\begin{proof}
 Recall that the operator 2-norm $\|\cdot\|_2$ is
$\|M\|_2 = \sup_{x\neq 0}{\|Mx\|_2/\|x\|_2} = \sqrt{\sigma_{\max}( M^\top M )} =   \sqrt{\sigma_{\max}( M M^\top )}=\|M^\top\|_2.$
where $\sigma_{\max}$ denotes the largest singular value. Hence, $\|PM\|_2=\sup_{x\neq 0}{\|PMx\|_2/\|x\|_2} =\sqrt{ \sigma_{\max} ( M^\top P^\top P M ) }= \sqrt{ \sigma_{\max}(M^\top M)}=\|M\|_2.$
as $P^\top P = I$. Using the fact that $\|M\|_2=\|M^\top\|_2$ for all $M\in \reals^{n\times n}$, as well as that $PP^\top = I$, we can show that $\|MP\|_2=\|P^\top M^\top\|_2=\|M^\top\|_2=\|M\|_2$.

The Frobenius norm is 
$\|M\|_F =  \sqrt{\trace(M^\top M)} = \sqrt{\trace(MM^\top)}=\|M^\top\|_F,$
hence
$\|PM\|_F = \sqrt{\trace(M^\top P^\top PM)}= \sqrt{\trace(M^\top M)} = \|M\|_F $
and, as in the case of the operator norm, we can similarly show $\|MP\|_F=\|P^\top M^\top\|_F=\|M^\top\|_F=\|M\|_F$. 
\end{proof}

In both norm cases, for $A,B \in \Symmetric^n$, we can
compute $d_S$ using a simple spectral decomposition.
Let $A = U \Sigma_A U^T$ and $B = V \Sigma_B V^T$ be the spectral decomposition of $A$ and $B$. As $A$ and $B$ are real and symmetric, we can assume
$U,V \in \Ortho^n$. Recall
that $U^{-1} = U^\top$ and $V^{-1} = V^\top$, while
$\Sigma_A$ and $\Sigma_B$ are diagonal and contain the eigenvalues of $A$ and $B$ sorted in increasing order; this orderning matters for computations below. 

The following theorem establishes that this decomposition readily yields the distance $d_S$, as well as the optimal orthogonal matrix $P^*$, when $\|\cdot\|=\|\cdot\|_F$: 
\begin{theorem}[\cite{umeyama1988eigendecomposition}]\label{th:fro_norm_d_S_comp}
$d_S(A,B) \triangleq \min_{P\in S}\|AP-PB\|_F=\|\Sigma_A - \Sigma_B\|_F$ and the minimum is attained by $P^* = UV^\top$.
\end{theorem}

\begin{proof}

The proof makes use of the following lemma by \cite{hoffman1953variation}.
\begin{lemma}\label{th:aux_lemma_eigen_values}
If $A$ and $B$ are Hermitian matrices
with eigenvalues $a_1 \leq a_2 \leq ...\leq a_n$ and $b_1 \leq b_2 \leq ...\leq b_n$ then 
\begin{equation}
\|A-B\|^2_F \geq {\sum^n_{i =1} (a_i - b_i)^2}
\end{equation}
\end{lemma}
\begin{rmk}
Note that if $\Sigma_A$ and $\Sigma_B$ are diagonal matrices
with the ordered eigenvalues of $A$ and $B$ in the diagonal, then Lemma \ref{th:aux_lemma_eigen_values} can be written as
$\|A-B\|_{F} \geq \| \Sigma_A - \Sigma_B\|_{F}$.
\end{rmk}
For any $P \in \Ortho^n$ and $\|\cdot\|=\|\cdot\|_F$ we have
\begin{align*}
\|AP-PB\|& = \|(A-PBP^{-1})P\| \stackrel{\text{Lem.}~\ref{eqlemma}}{=} 
\|A-PBP^\top\| \\&= \|U \Sigma_A U^\top-PV \Sigma_B V^\top P^\top \|\\
&=\|U (\Sigma_A -U^\top P V \Sigma_B V^\top P^\top U)U^\top \|\\
& \stackrel{\text{Lem.}~\ref{eqlemma}}{=} \| \Sigma_A -U^\top P V \Sigma_B V^\top P^\top U \|\\
& = \| \Sigma_A -\Delta \Sigma_B \Delta^\top\|
\end{align*}
where we define $\Delta \triangleq U^\top P V$.
As a product of orthogonal matrices, $\Delta \in \Ortho^n$.
Notice that
\begin{align*}
 \| \Sigma_A -\Delta \Sigma_B &\Delta^\top\| =\\&= \| \Sigma_A -\Delta \Sigma_A \Delta^\top + \Delta (\Sigma_B  -\Sigma_A )\Delta^\top\| \\
 &\leq \| \Sigma_A -\Delta \Sigma_A \Delta^\top\| + \|\Delta (\Sigma_B  -\Sigma_A )\Delta^\top\|\\
& \stackrel{\text{Lem.}~\ref{eqlemma}}{=} \|\Sigma_A -\Delta \Sigma_A \Delta^\top\| + \| \Sigma_B  -\Sigma_A \|.
\end{align*}
Therefore, for any $P \in \Ortho^n$,
$ \|\Sigma_A - \Sigma_B\|\leq d_S(A,B) \leq \|\Sigma_A -\Delta \Sigma_A \Delta^\top\| + \| \Sigma_B  -\Sigma_A \|,$
where the first inequality follows by Lemma  \ref{th:aux_lemma_eigen_values}
if we notice that $\|AP - PB\| = \|A - PBP^{-1}\|$ and that $PBP^{-1}$ and $B$ have the same spectrum for any $P$.
If we choose $P = U V^\top$ then $\Delta = I$ and the result follows.
\end{proof}

We can compute $d_S$ when $S= \Ortho^n$ and $\|\cdot\|$ is the operator norm in the exact same way.

\begin{theorem}\label{th:OP_norm_d_S_comp}
Let $\|\cdot\| = \|\cdot\|_{2}$ be the operator 2-norm.
Then, $d_S(A,B) \triangleq \min_{P\in S}\|AP-PB\|_2=\|\Sigma_A - \Sigma_B\|_2$ and the minimum is attained by $P^* = UV^\top$.
\end{theorem}
\begin{proof}
The proof follows the same steps as the proof of Theorem \ref{th:fro_norm_d_S_comp}, using  Lemma \ref{th:aux_lemma_eigen_values_op_norm} below instead of Lemma \ref{th:aux_lemma_eigen_values}.

\begin{lemma}\label{th:aux_lemma_eigen_values_op_norm}
If $A$ and $B$ are Hermitian matrices
with eigenvalues $a_1 \leq a_2 \leq ...\leq a_n$ and $b_1 \leq b_2 \leq ...\leq b_n$ then 
\begin{equation}
\|A-B\|_{2} \geq \max_{i} |a_i - b_i|.
\end{equation}
\end{lemma}
\begin{rmk}
Note that if $\Sigma_A$ and $\Sigma_B$ are diagonal matrices
with the ordered eigenvalues of $A$ and $B$ in the diagonal, then Lemma \ref{th:aux_lemma_eigen_values_op_norm} can be written as
$\|A-B\|_{2} \geq \| \Sigma_A - \Sigma_B\|_{2}$.
\end{rmk}
\begin{proof}[Proof of Lemma \ref{th:aux_lemma_eigen_values_op_norm}]
Let $\tilde{B} = -B$ have eigenvalues
$\tilde{b}_1 \leq \tilde{b}_2 \leq ...\leq \tilde{b}_n$ and let $C = A + \tilde{B}$ have eigenvalues $c_1 \leq c_2 \leq ...\leq c_n$.
We make use of the following lemma by Weyl \cite{horn2012matrix} to lower bound $c_n$.
\begin{lemma}\label{lem:weyl} If $X$ and $Y$ are Hermitian with eigenvalues $x_1 \leq...\leq x_n$ and $y_1 \leq ... \leq y_n$ 
and if $X+Y$ has eigenvalues
$w_1 \leq ... \leq w_n$ then
$x_{i-j+1} + y_j \leq w_i$ for all $i = 1,\ldots,n$ and $ j =1, \ldots , i$.
\end{lemma}

If we choose $X = \tilde{B}$, $Y = A$ and $i= n$ we get
 $a_{j} + \tilde{b}_{n+1-j} \leq c_n$ for all $j = 1, \ldots, n$.

Since $\tilde{b}_{n+1-j} = -{b}_{j}$  we get that $a_{j} -{b}_{ j} \leq c_n$, for any $j$.
Similarly, by exchanging the role of $A$ and $B$, we can lower bound the largest eigenvalue
of $B-A$, say $d_n$, by $b_j -a_{j}$ for any $j$.
Notice that, by definition of the operator norm and the fact that $A-B$ is
Hermitian,  $\|A-B\|_{2} \geq |c_n|$ and $\|B-A\|_{2} \geq |d_n|$. Since $\|B-A\|_{2} =\|A-B\|_{2}$ we have that $\|A-B\|_{2} \geq \max\{|c_n|,|d_n|\} \geq \max\{c_n,d_n\}\geq \max\{a_j - b_j,b_j - a_j\} = |a_j - b_j|$ for all $j$.
Taking the maximum over $j$ we get that $\|A-B\|_{2}  \geq \max_j |a_j - b_j|$, and the lemma follows.
\end{proof}
The proof of Thm.~\ref{th:OP_norm_d_S_comp} proceeds along the same steps as the above proof, using again the fact that, by Lemma~\ref{eqlemma},  $\|M\|_2 = \|MP\|_2 = \|PM\|_2$ for any $P \in \Ortho^n$ and any matrix $M$, along with Lemma~\ref{lem:weyl}.
\end{proof}

\section{The Weisfeiler-Lehman (WL) Algorithm.}\label{sec:WL}
 The WL algorithm \citep{weisfeiler1968reduction}  is a 
graph isomorphism heuristic. To gain some intuition on the algorithm, note that two isomorphic graphs must have the same degree distribution. More broadly, the distributions of  $k$-hop neighborhoods in the two graphs must also be identical.
Building on this, to test if two undirected, unweighted graphs are isomorphic, WL  \emph{colors} the nodes of a graph $G(V,E)$ iteratively. At iteration $0$, each node $ v\in V$ receives the same \emph{color} $c^0(v):=1$.
Colors at iteration $k+1\in\naturals$ are defined recursively via 
$c^{k+1}(v):= \mathsf{hash}\left(\mathsf{sort}\left( \textsf{clist}^k_v\right) \right)$where $\mathsf{hash}$ is a perfect hash function, and
$\textsf{clist}^{k}_v=[c^{k}(u):(u,v)\in E)]$
is a list containing the colors of all of $v$'s neighbors at iteration $k$. Intuitively, two nodes in $V$ share the same color after $k$ iterations  if their $k$-hop neighborhoods are isomorphic. WL terminates when the partition of $V$ induced by colors is stable from one iteration to the next.
This coloring extends to weighted directed graphs by appending weights and directions to colors in $\textsf{clist}_{v}^k$. After coloring two graphs $G_A,G_B$, WL declares a non-isomorphism if their color distributions differ. If not, then they \emph{may} be isomorphic and WL gives a set of \emph{constraints} on candidate isomorphisms: a permutation $P$ under which $AP=PB$  \text{must} map nodes in $G_A$ to nodes in $G_B$ of the same color.

\section{Algorithms and Implementation Details}\label{app:alg}
We outline here additional impementation details about the algorithms summarized in Table~\ref{table:comp}.
\begin{packeditemize}
\item \NetAlignBP, \IsoRank, \SparseIsoRank and \NetAlignMR, for which code is publicly available \citep{bayatticode}, are described by \cite{bayati2009algorithms}. \Natalie is described in \citep{el2015natalie}; code is again available \citep{nataliecode}. All five algorithms output $P \in \Permutations^n$.
\item The algorithm in \citep{lyzinski2016graph} outputs one $P \in \Permutations^n$ and one
$P' \in \DStoch^n$. We use $P \in \Permutations^n$ to compute $\|AP - PB\|_1$ and call this \InnerPerm. We use $P' \in \DStoch^n$ to compute $\|AP' - P'B\|_1$ and $\|AP' - P'B\|_2$ and call these algorithms \InnerDSLone and \InnerDSLtwo respectively. We use our own CVX-based projected gradient descent solver for the non-convex optimization problem the authors propose.
\item \DSLone and \DSLtwo denote $d_S(A,B)$ when $S \in \DStoch^n$ and $\|\cdot\|$   is $\|\cdot\|_1$ (element-wise) and $\|\cdot\|_F$, respectively. We implement them in Matlab (using CVX) as well as in C, aimed for medium size graphs and multi-core use. We also implemented a distributed version in Apache Spark~\citep{zaharia2010spark} that scales to very large graphs over multiple machines based on the Alternating Directions Method of Multipliers \citep{boyd2011distributed}. \item {\bf ORTHOP} and {\bf ORTHFR} denote
$d_S(A,B)$ when $S \in \Ortho^n$ and $\|\cdot\|$ is $\|\cdot\|_2$ (operator norm) and $\|\cdot\|_F$ respectively.
We compute them using an eigendecomposition
(\alt{See Appendix \ref{sec:umeyama} in Supplement}{See Appendix \ref{sec:umeyama}}).

\item For small graphs, we compute $d_{\Permutations^n}(A,B)$ using our brute-force GPU-based code.  For a single pair of graphs with $n \geq 15$ nodes, \EXACT already takes several days to finish. For $\|\cdot\| = \|\cdot\|_1$ in $d_S$ (element-wise or matrix norm), we have implemented the chemical distance as an integer value LP and solved it using branch-and-cut. It did not scale well for $n \geq 15$.

\item We implemented the WL algorithm over Spark to run, multithreaded, on a machine with 40 CPUs.  \end{packeditemize}
We use all public algorithms as black boxes with their default parameters, as provided by the authors.

\section{Graphs of Different Sizes}\label{app:size}

For simplicity, we described our framework for graphs of equal sizes.
However, we can extended in different ways to produce a metric
for graphs of different sizes. 
These extensions all start by extending two graphs, $G_A$ and $G_B$, with dummy nodes
such that the new graphs $G'_A$ and $G'_B$ have the same number of nodes. If $G_A$ has
$n_A$ nodes and $G_B$ has $n_B$ nodes we can, for example, add $n_B$ dummy nodes to $G_A$ and $n_A$
dummy nodes to $G_A$.
Once we have $G'_A$ and $G'_B$ of equal size, we can use the methods we already described to compute
a distance between $G'_A$ and $G'_B$ and return this distance as the distance between $G_A$ and $G_B$.

The different ways of extending the graphs differ in how the dummy nodes connect
to existing graph nodes, how dummy nodes connect to themselves, and what kind of penalty
we introduce for associating dummy nodes with existing graph nodes.
\emph{Method 1:} One way of extending the graphs is to add dummy nodes and leave them isolated, i.e.,
with no edges to either existing nodes or other dummy nodes. 
Although this might work when both graphs are dense,
it might lead to non desirable results when one of the graphs is sparse. For example, let $G_A$ be $3$ isolated nodes and
$G_B$ be the complete graph on $4$ nodes minus the edges
forming triangle $\{(1,2),(2,3),(3,1)\}$. Let us assume that $S = 
\Permutations^n$, such that, when we compute the distance between
$G_A$ and $G_B$, we produce an alignment between the graphs. One desirable outcome would be for $G_A$ to be aligned
with the three nodes in $G_B$ that have no edges among them.
This is basically solving the problem of finding a sparse subgraph inside a dense graph.
However, computing $d_S(A',B')$, where $A'$ and $B'$ are the extended adjacency matrices, could equally well align $G_A$
with the $3$ dummy node of $G'_B$. 
\emph{Method 2:} Add dummy nodes and connect each dummy node to all existing nodes and all other dummy nodes. This avoids the
issue described for method 1 but creates a similar non desirable situation: since the dummy nodes in each extended graph form a click, we might align $G_A$, or $G_B$, with just dummy nodes, instead of producing an alignment between existing nodes in $G_A$ and existing nodes in $G_B$. 
\emph{Method 3:} If both $G_A$ and $G_B$ are unweighted graphs, a method that avoids both issues above (aligning a sparse graph with isolated dummy nodes or aligning a dense graphs with clicks of dummy nodes) is to connect each dummy node to all existing nodes and all other dummy nodes with edges of weight $1/2$. This method works because, when $S = \Permutations^n$, it discourages alignments of pairs existing-existing nodes in $G_A$ with pairs  dummy-dummy nodes or pairs dummy-existing nodes in $G_B$, and vice versa.
\emph{Method 4:} One can also discourage aligning existing node with dummy nodes by introducing a linear term as in \eqref{addlocal}.

\end{document}